\documentclass{amsart}

\usepackage{amssymb, amsmath, amsfonts,mathrsfs,graphicx ,mathdots, amscd}
\usepackage{color}
\usepackage{mathtools}
\usepackage{lineno}

 \usepackage{youngtab}
\usepackage[boxsize=.3 em]{ytableau}

\usepackage[all,cmtip]{xy}
\usepackage{stmaryrd}
\newtheorem{thm}{Theorem}[section]
\newtheorem{prop}[thm]{Proposition}
\newtheorem{lemma}[thm]{Lemma}

\newtheorem{cor}[thm]{Corollary}
\newtheorem{defn}[thm]{Definition}
\newtheorem{example}[thm]{Example}
\newtheorem{remark}[thm]{Remark}

\newtheorem{qus}[thm]{Question}

\newtheorem{conjecture}[thm]{Conjecture}

 \numberwithin{equation}{section}

\newcommand{\comment}[1]{}

\newcommand{\bull}{{\scriptscriptstyle \bullet}}

\newcommand{\PP}{{\mathbb P}}

\newcommand{\Fl}{\mathbb{ F}\ell}
\newcommand{\Flndot}{\Fl(\ndot)}

\newcommand{\ndot}{n_\bull}

\renewcommand{\b}{\hat b}
\renewcommand{\u}{\hat u}
\renewcommand{\v}{\hat v}

\renewcommand{\j}{{\mathrm j}}

\newcommand{\inv}{^{-1}}

\begin{document}{\allowdisplaybreaks[4]

\title[A Pl\"ucker coordinate mirror for partial flag varieties]{A Pl\"ucker coordinate mirror for partial flag varieties and quantum Schubert calculus}


\author{Changzheng Li}
 \address{School of Mathematics, Sun Yat-sen University, Guangzhou 510275, P.R. China}
\email{lichangzh@mail.sysu.edu.cn}

\author{Konstanze Rietsch}
\address{ Department of Mathematics, King's College London, Strand, London WC2R 2LS, UK}
\email{konstanze.rietsch@kcl.ac.uk}

\author{Mingzhi Yang}
 \address{School of Mathematics, Sun Yat-sen University, Guangzhou 510275, P.R. China}
\email{yangmzh8@mail2.sysu.edu.cn}

\author{Chi Zhang}
 \address{Department of Mathematics, Caltech, 1200 East California Boulevard, Pasadena, CA 91125}
\email{czhang5@caltech.edu}

\thanks{}

\date{
      }




\begin{abstract}
We construct a Pl\"ucker coordinate superpotential $\mathcal{F}_-$ 
that is mirror to a partial flag variety $\Flndot$. Its Jacobi ring recovers the small quantum cohomology of $\Flndot$, and we prove a folklore conjecture in mirror symmetry. Namely, we show that the eigenvalues for the action of the first Chern class  $c_1(\Flndot)$ on quantum cohomology are equal to the critical values of $\mathcal{F}_-$. We achieve this by proving new identities in quantum Schubert calculus that are inspired by our formula for $\mathcal{F}_-$ and the mirror symmetry conjecture.
 \end{abstract}

\maketitle

\tableofcontents
\section{Introduction}

Mirror symmetry is a fascinating   phenomenon arising in string theory:  two  apparently completely  different objects on  A-model and B-model give rise to equivalent physics. 
 Mathematical descriptions of mirror symmetry, in terms of equivalence of mathematical structures, were first
made for pairs of Calabi-Yau manifolds   in early 1990s (see e.g. \cite{HKKMS}). The (closed string) mirror symmetry was extended to Fano manifolds $X$ on the  topological A-model soon after by Givental \cite{Givental95, Givental98} and Eguchi-Hori-Xiong \cite{EHX}.
In this case, the topological B-model is given by a Landau-Ginzburg model $(\check X, W)$,  consisting of
a non-compact K\"ahler manifold $\check X$ and a holomorphic function $W: \check X\to \mathbb{C}$ called the superpotential.
 Mirror symmetry predicts equivalences between both sides on various levels. For instance on one level,  the (small) quantum cohomology ring $QH^*(X)$ should be isomorphic to the Jacobi ring $Jac(W)$ of $W$.

Studying mirror symmetry for $X$ a priori requires a good construction of the mirror superpotential $W$. However, this is only known for certain Fano manifolds, with toric Fano manifolds and complete intersections inside toric manifolds being typical examples, following work of Givental \cite{Givental95, Givental98} and Hori-Vafa \cite{HV}.
In this article, we will focus on the case when $X=\Flndot$ is a partial flag variety parameterizing  flags of quotient vector subspaces of $\mathbb{C}^n$.
          Special cases include complex Grassmannians $Gr(k, n)$ and complete flag variety $\mathbb{F}\ell_n$.
     Candidate Landau-Ginzburg models  for  $Gr(k, n)$ and   $\mathbb{F}\ell_n$ were constructed by Eguchi-Hori-Xiong \cite{EHX} and Givental \cite{Givental97} respectively. They were later generalized to $\Flndot$ by Batyrev-Ciocan-Fontanine-Kim-van Straten \cite{BCFKS}. See also \cite{NNU} for a construction using  holomorphic disk counts.  Here different approaches turned out to result in identical versions of the superpotential, namely arriving at a particular Laurent polynomial  $W_{\rm tor}$ defined on a complex torus of dimension $\dim \Flndot$. It turned out that there is a toric degeneration of  $\Flndot$ with the central fiber a singular toric variety $X_{0}$, and the superpotential $W_{\rm tor}$ coincides with the superpotential mirror to $X_0$ as constructed by Givental and Hori-Vafa. This superpotential, however, has a disadvantage as mirror for $\Flndot$ in that its Jacobi ring does not fully recover quantum cohomology. For example, for $Gr(2, 4)$ the mirror superpotential should ideally have $6=\dim H^*(Gr(2,4))$ critical points, but $W_{\rm tor}$ only has $4$.

     In \cite{Rie08}, the second-named author wrote down a Lie theoretical superpotential, namely a function $\mathcal{F}_{\rm Lie}: Z_P\to \mathbb{C}$ defined on a subvariety $Z_P$ of $B_-$.
     Here $G$ is a connected complex reductive Lie group, and $P$ is a parabolic subgroup of $G$ containing a Borel subgroup $B_-$. This function had appeared separately earlier in a different context, as part of a theory of geometric crystals \cite{BK:GeometricCrystalsII}.
It is shown in \cite{Rie08} that the fiberwise critical locus of $\mathcal{F}_{\rm Lie}$ is isomorphic to (an open dense part of) the so-called Peterson variety stratum $Y_P$ in the flag variety $G/B_-$.  This relates the superpotential $\mathcal{F}_{\rm Lie}$ to quantum cohomology via the remarkable isomorphism  of Dale Peterson's, described in his unpublished lecture notes \cite{Pert}, between $\mathbb{C}[Y_P]$ and the small quantum cohomology ring $QH^*(G^\vee/P^\vee)$ of the Langlands dual flag variety. A proof of Peterson's isomorphism for the type $A$ case, that is for  $G^\vee/P^\vee=\Flndot$, was given in \cite{Rie03}. Some other cases were covered in \cite{Cheo, LaSh}, and the general case was proved in a recent preprint \cite{Chow22}. The combination of both isomorphisms leads to mirror symmetry for flag varieties on the level of small quantum cohomology. Namely, the ring $QH^*(G^\vee/P^\vee)$, with inverse quantum parameters adjoined, is isomorphic to the (fiberwise) Jacobi ring of $\mathcal F_{\rm Lie}$.

    This is not the end of the story, but only the end  of the beginning.
The function $\mathcal{F}_{\rm Lie}$ is defined quite indirectly, and it is desirable to find a compact expression in terms of coordinates on the mirror space $Z_P$. In the special case of  $\Flndot=Gr(n-k, n)$, a natural isomorphic interpretation of the mirror space was given in \cite{MaRi}. There, $Z_P$ was identified with a trivial family over $\mathbb{C}^*$ with fiber a particular open log Calabi-Yau subvariety in the Langlands dual Grassmannian $Gr(k, n)$. Moreover, \cite{MaRi} gave a very compact and clean expression for  $\mathcal{F}_{\rm Lie}$ using the Pl\"ucker coordinates of $Gr(k, n)$. This also led to an improved mirror symmetry result on the higher level of $D$-modules.

         One generalization of this $Gr(n-k, n)$ construction is to  cominuscule Grassmannians of other types. The fiber of the mirror space is then inside the Langlands dual minuscule Grassmannian, which has (generalized)  Pl\"ucker coordinates, due to its embedding into the projective space of a minuscule representation. Corresponding coordinate presentation of $\mathcal{F}_{\rm Lie}$ have been individually obtained for
          quadrics \cite{PRW},  Lagrangian Grassmannians \cite{PeRi13},  the Cayley plane and the Freudenthal variety \cite{SpWa}.

            The generalization of $Gr(n-k, n)$ of interest to us here is the partial flag variety $X=\Flndot$. As the first main result of this paper, we provide a Pl\"ucker coordinate formula version $\mathcal{F}_-$ of the superpotential $\mathcal{F}_{\rm Lie}$ for this case. To construct the domain we consider the Langlands dual partial flag variety $F\ell_{\ndot}=F\ell_{n_1,  \cdots, n_r; n}=P\!\setminus\! G$ that parameterizes flags of vector subspaces $V_{n_j}$ in the dual vector space of $\mathbb{C}^n$.
  Let $(P\!\setminus\! G)^\circ$ denote the complement of the Knutson-Lam-Speyer anti-canonical divisor $-K_{F\ell_{\ndot}}$ \cite{KLS}, which consists of $(n-1+r)$ irreducible components (see Proposition \ref{propanti}).
\begin{thm}\label{intro:Fminus}
   There is an isomorphism
    \[
\psi_-:Z_P   \longrightarrow (P\backslash G)^\circ \times  \prod\limits_{i\in I^P} \mathbb{C}^*_q,\qquad \mbox{ where }  I^P:=\{n_1, \cdots, n_r\},
\]
constructed in \eqref{e:psi-}. The superpotential $
\mathcal F_-:=\mathcal F_{\operatorname{Lie}}\circ \psi_-\inv:(P\backslash G)^\circ\times  \prod\limits_{i\in I^P} \mathbb{C}^*_q \to \mathbb C$ consists of $(n-1+r)$ summands, 
 \begin{align*}\mathcal F_-(Pz,\mathbf q)=   \sum\limits_{i\in I^P} q_i v_{i,i+1}+ \sum\limits_{i=1}^{n-1}u_{i,i+1}.
 \end{align*}
The summands satisfy
   \begin{enumerate}
     \item the $v_{i, i+1}$ are all of the form ${p_{J'}\over p_J}$ for some Pl\"ucker coordinates;
     \item  the $u_{i, i+1}$ are of the form ${p_{J'}\over p_J}$ if $i\in I^P$, or if $1\le i\le n_1$ or $n_r\le i\le n-1$.  Otherwise, if $n_j<i<n_{j+1}$ for some $j\in\{1, \cdots, r-1\}$, then
 $u_{i, i+1}$ is of the form ${f_1\over f_2}$ with each $f_i$ a quadratic polynomial in the Pl\"ucker coordinates;
     \item all $v_{i, i+1}$  and $u_{i, i+1}$ have pole of order 1 along a (unique) irreducible component   of  $-K_{F\ell_{\ndot}}$.
   \end{enumerate}
\end{thm}
\noindent In fact, the denominators in the summands of $\mathcal{F}_-$ are precisely the defining equations of the irreducible components of $-K_{F\ell_{\ndot}}$ \cite{LSZ}.

The isomorphism $\Psi_-$ will be constructed explicitly in Definition \ref{defF-}. Explicit  expressions for $v_{i, i+1}$  and $u_{i, i+1}$ will be given in \textbf{Theorem \ref{fminus}}.
 Here we provide an example to give a first impression.
\begin{example}\label{ex:F247}
   For $F\ell_{n_\bullet}=F\ell_{2, 4; 7}\hookrightarrow \mathbb{P}(\bigwedge^2\mathbb{C}^7)\times \mathbb{P}(\bigwedge^4\mathbb{C}^7)$, we have
   \begin{align*}
     \mathcal{F}_-&=q_2{p_{46}\over p_{67}}+q_4{p_{1467}\over p_{4567}}+{p_{27}\over p_{17}}+{p_{24}p_{1567}-p_{14}p_{2567}+p_{12}p_{4567}\over p_{23}p_{1567}-p_{13}p_{2567}+p_{12}p_{3567}}+{p_{2346}\over p_{2345}}+{p_{3457}\over p_{3456}}
             +{p_{13}\over p_{12}}+{p_{1235}\over p_{1234}}.
       \end{align*}
 \end{example}
 \begin{remark}
     A straightforward generalization of the superpotential in \cite{MaRi} leads to another superpotential $\mathcal{F}_+$ defined on $(P_+\backslash G)^\circ\times  \prod_{i\in I^P} \mathbb{C}^*_q$. The Pl\"ucker coordinate expression of  $\mathcal{F}_+$, however, appears to be complicated and does not have the similar good properties (especially property (3) above). We refer to Examples \ref{ex:FullFlagF+} and \ref{ex:FullFlagF-} for a comparison of $\mathcal{F}_+$ and $\mathcal{F}_-$ in the case of the complete flag variety $F\ell_3$.
  \end{remark}

\begin{remark} There is a special embedding of the domain of  the \cite{BCFKS} Laurent polynomial mirror, $W_{tor}$, into $\Fl_{n_\bullet}$ constructed in \cite{Rie:Nagoya}. Conjecturally, $W_{tor}$ should be the pullback of $\mathcal F_{\rm{Lie}}$ under this embedding. This conjecture is known to hold in the complete flag variety and the Grassmannian cases \cite{Rie08,MaRi}. For more general partial flag varieties the conjecture is not proved, but it is shown in \cite{Rie:Nagoya} that $W_{tor}$ and the pullback of $\mathcal F_{\rm{Lie}}$ have the same fiberwise critical points. This conjecture now translates into a conjecture about $\mathcal F_-$ via the isomorphism between $\mathcal F_-$ and $\mathcal F_{\rm {Lie}}$. 
\end{remark}

  \begin{remark}
 In \cite{GuSh},  Gu and Sharpe proposed a construction of `non-abelian' mirrors, examples of which included $\Flndot$. Their approach is closely related to \cite{HV} but involves more variables than the dimension of $\Flndot$ (see also \cite{GuKa}). Their mirror diverges already in the Grassmannian case from the mirror constructions \cite{EHX,Rie08,MaRi} that are related to ours here, see \cite[Section~4.9]{GuSh}.

     Another mirror construction inspired by viewing flag varieties as non-abelian GIT quotients was given by Kalashnikov \cite{Kala}.  Namely, Kalashnikov proposed a generalization of the superpotential from \cite{MaRi} for  $Gr(n-k, n)$  to partial flag varieties $\Flndot$ in the form of a rational function on a product of Grassmannians, expressed explicitly in terms of Pl\"ucker coordinates. The Kalashnikov superpotential has the advantage that it recovers the aforementioned Laurent polynomial $W_{\rm tor}$ in a cluster chart, and is directly related to a toric degeneration of $\Flndot$. Kalashnikov also described a relation (on the level of critical points) between her superpotential and Gu-Sharpe's superpotential in a special case.

To compare the Kalashnikov formula with our $\mathcal F_-$, consider the partial flag variety $\Flndot=\mathbb{F}\ell(4; 2, 1)$. Kalashnikov's superpotential is a rational function on $Gr(2, 4)\times Gr(1, 2)\times (\mathbb{C}^*)^2$ described in terms of Pl\"ucker coordinates $[p_{ij}; \hat p_k]$ by
       \[
W_{\rm Kal}=\frac{p_{13}}{p_{12}}+ \frac{p_{24}+q_2}{p_{23}}+ \frac{p_{24}}{p_{14}}+ \frac{q_2 p_{13} \hat p_2}{p_{34}}+ \frac{\hat p_2}{\hat p_1}+ \frac{q_1}{\hat p_2}.
\]
 Our superpotential $\mathcal{F}_-$ is a rational function on  $F\ell_{1,2; 4}\times (\mathbb{C}^*)^2$, given in  terms of  $[p_k; p_{ij}]$ by
   \[
\mathcal{F}_{-}=q_1 \frac{p_{3}}{p_{4}}+ q_2 \frac{p_{14}}{p_{34}}+\frac{p_{2}}{p_{1}}+\frac{p_{13}}{p_{12}}+\frac{p_{24}}{p_{23}}.
\]
The superpotential $W_{\rm Kal}$ is always a positive Laurent polynomial in Pl\"ucker coordinates, while ours in general is not (neither Laurent, nor positive). However, for certain special values of the parameters $q_i$ the superpotential $W_{\rm Kal}$ does not have the full set of critical points, as was also pointed out in \cite{Kala}.
Namely, in the above example, $\mathcal{F}_{-}$ has $12$ critical points along its $q_1=q_2=1$ fiber, in agreement with $\dim H^*(\mathbb{F}\ell(4; 2, 1))=12$. One of these critical points, the one with critical value $-3$, is not visible for $W_{\rm Kal}|_{\mathbf{q}=(1, 1)}$.
  \end{remark}

Let us now recall that, on the A-side, the (small) quantum cohomology ring $QH^*(X)=(H^*(X,\mathbb C)\otimes \mathbb{C}[\mathbf{q}], \cdot)$ of the Fano manifold $X$ is a deformation of the classical cohomology ring $H^*(X, \mathbb{C})$ by incorporating genus zero, 3-point Gromov-Witten invariants. The quantum multiplication by the first Chern class of $X$ induces a linear operator
          $$\hat c_1(\mathbf{q}): QH^*(X)\longrightarrow QH^*(X); \,\, \beta\mapsto c_1(X)\cdot \beta$$
          depending on the values of the deformation parameters $\mathbf q=(q_i)_i$, also called quantum parameters.
  Here we treat the $q_i$ as nonzero complex numbers, so that $QH^*(X)=H^*(X)$ as vector spaces. On the B-side, we consider the superpotential $W=W_{\mathbf{q}}$
   with the quantum parameters fixed correspondingly.
  Now let us state a celebrated folklore conjecture in mirror symmetry.
    \begin{conjecture}\label{conjms}
  The eigenvalues of the first Chern class operator $\hat c_1(\mathbf{q})$ coincide with the critical values of the mirror  superpotential $W_{\mathbf{q}}$.
  \end{conjecture}
\noindent There has been very  little progress on this conjecture in the past two decades.  The case of toric Fano manifolds was   first proved by Auroux \cite{Auroux}, which was also known to Kontsevich and Seidel.   Recently, Yuan \cite{Yuan} proved that  the critical values of the family Floer mirror Landau-Ginzburg superpotential are the eigenvalues of the first Chern class, under certain assumptions. The  cases of complex Grassmannians and quadrics were proved implicitly in \cite{MaRi} and \cite{Hu} respectively.

  As a central result of this paper, we prove a theorem that implies this conjecture for any partial flag variety $X=\Flndot$.

Let us write  $\mathbf{q}=(q_{n_1}, \cdots, q_{n_r})$ for the quantum parameters associated to $\Flndot$, and view them as coordinates on an algebraic torus that we denote by $ \prod_{i\in I^P} \mathbb{C}^*_q$. Let us consider the (fiberwise) Jacobi ring,
\begin{equation}\label{e:Jacin}
Jac(\mathcal F_-):=\mathcal O\left((P\backslash G)^\circ\times \prod_{i\in I^P} \mathbb{C}^*_q\right)/(\partial_{(P\backslash G)^\circ} \mathcal{F}_-),
\end{equation}
where  we are taking partial derivatives of $\mathcal F_-$ in the $(P\backslash G)^\circ$ directions only. Using Theorem~\ref{intro:Fminus}, and the isomorphism between fiberwise Jacobi ring of $\mathcal F_{\rm Lie}$ and quantum cohomology resulting from \cite{Pert,Rie03,Rie08}, we obtain an isomorphism of rings
\begin{equation}\label{e:Theta}\Theta: Jac(\mathcal{F}_-)\overset{\sim}\longrightarrow QH^*(X)[q_{n_1}^{-1},\cdots,q_{n_r}^{-1}].
\end{equation}
See Section~\ref{subseccrit} for a more detailed description. We can now state our second main theorem.

 \begin{thm} \label{intro:firstchernclass}
  For the class $[\mathcal{F}_-]$    of   $\mathcal{F}_-$ in the Jacobi  ring $Jac(\mathcal{F}_-)$, we have
  $$\Theta([\mathcal{F}_-])={c_1(X)},$$
where $c_1(X)$ is the first Chern class of $X=\Flndot$, as element of the small quantum cohomology ring.
\end{thm}
\noindent The above theorem is stated again in an isomorphic form in \textbf{Theorem \ref{firstchernclass}}, using a version $\mathcal F_R$ of the superpotential whose domain relates more directly to the Peterson variety.
By interpreting the critical values of  $\mathcal{F}_-$ as eigenvalues  for the operator of multiplication by $[\mathcal{F}_-]$ on $Jac(\mathcal{F}_-)$, we obtain the following corollary.
  \begin{cor}
     Conjecture \ref{conjms} holds for $X=\Flndot$ and the mirror superpotential $\mathcal{F}_{-}$.
  \end{cor}
 We note that isomorphically changing the domain of the superpotential does not affect the critical values. Therefore the same corollary holds for $\mathcal{F}_{\rm Lie}$, and $\mathcal F_R$. We also note that $\mathcal{F}_R$ and $\mathcal{F}_-$ look to be related by the chiral map in \cite{GaLa}.

  An  exciting aspect of this part of our paper is the interaction between the Conjecture~\ref{conjms} in mirror symmetry and identities in quantum Schubert calculus. The quantum cohomology ring $QH^*(\Flndot)$ has a $\mathbb{C}[\mathbf{q}]$-basis of Schubert classes $\sigma_w$, that is indexed by permutations in $S_n$ with descents at most in $n_j$, for $j\in\{1, \cdots, r\}$. The study of the ring structure of $QH^*(\Flndot)$ in terms of this basis, referred to as (type $A$) quantum Schubert calculus, is an area of great independent interest from the viewpoint of enumerative geometry. One of the  most central problems is to find a manifestly positive formula for the Schubert structure constants in the quantum product of two Schubert classes. Another topic of interest is the study of identities among quantum products of Schubert classes. For example, the quantum Schubert polynomials \cite{FominGelfandPostnikov, Fontanine} are expressions for general Schubert classes as polynomials in special Schubert classes. The quantum Giambelli formula \cite{Bert} for complex Grassmannians is another example. It turns out, that mirror symmetry also helps us find identities of this kind. Let us illustrate this from the perspective of the following natural question. Consider the isomorphism $\Theta$ from \eqref{e:Theta}
   \begin{qus}
     What are the preimages of the Schubert classes in $QH^*(\Flndot)$ under $\Theta$?
   \end{qus}
   Assuming the answer, one may expect to find relations in  $QH^*(\Flndot)$ simply by studying the mirror superpotential. Indeed, in the special case of $\Flndot=Gr(n-k, n)$,
our $\mathcal{F}_-$ turns out to coincide with the superpotential $\mathcal{F}_+$ of \cite{MaRi} (see Example \ref{ex:F-Gr}). Therefore, we have  $\Theta\inv(\sigma_w)=[p_w]$, where  $p_w$ denotes the (suitably normalised) Pl\"ucker coordinate corresponding to the  Grassmannian permutation $w$, as described in \cite{MaRi}. Each term in $\mathcal{F}_-$ turns out to reveal a quantum cohomology relation, see \cite[Remark 6.2]{MaRi}, recovering an instance of the known quantum Pieri-Chevalley formula in this case. For more general partial flag varieties the question above can be answered for certain Schubert classes using work of Peterson, see Section~\ref{s:peterson}. This means that quantum Schubert calculus relations involving these classes can be viewed as relations in the Jacobi ring.

The approach of using the superpotential for understanding quantum cohomology was used also in \cite{ChKa}. Namely, one can consider partial derivatives of the superpotential which naturally represent the zero class in the Jacobi ring, and translate these into quantum cohomology relations via the mirror isomorphism. Following this approach, \cite{ChKa} obtained certain relations involving `quantum hooks' via $W_{\rm Kal}$.

Our final result is a set of quantum Schubert calculus identities related to our formula for $\mathcal F_-$. The proof of Theorem~\ref{intro:firstchernclass},
turns out to involve showing each term in $\mathcal{F}_-$ corresponds to specific class in quantum cohomology. This requires certain relations to be proved in $QH^*(\Flndot)$.
    For instance, the term ${p_{24}p_{1567}-p_{14}p_{2567}+p_{12}p_{4567}\over p_{23}p_{1567}-p_{13}p_{2567}+p_{12}p_{3567}}$ in Example~\ref{ex:F247} relates to an identity
      $$ \sigma_{24}\cdot \sigma_{1567}-\sigma_{14}\cdot \sigma_{2567}+\sigma_{12}\cdot \sigma_{4567}= (\sigma_{23}\cdot \sigma_{1567}-\sigma_{13}\cdot \sigma_{2567}+\sigma_{12} \sigma_{3567}) \cdot (\sigma_{13}+\sigma_{1235})$$
      in quantum Schubert calculus.  Note that we have simplified the notations, for instance  by $\sigma_{24}$ above we mean the Schubert class $\sigma_{2413567}$ indexed by the Grassmannian permutation $2413567$ in one-line notation.
The above identity is equivalent to the following simpler one by using  quantum Chevalley-Monk formula \cite{Fontanine, Buch, Miha},
   \begin{align}\label{ex:iden}
    \sigma_{1526347}\cdot \sigma_{2314567}-\sigma_{2516347}\cdot\sigma_{1324567}+\sigma_{3516247}\cdot\sigma_{1234567}=0.
  \end{align}


Our final theorem, 
 that we prove concurrently with Theorem~\ref{intro:firstchernclass}, gives new relations in quantum Schubert calculus of $QH^*(\Flndot)$, including  identity \eqref{ex:iden} as one example.

\begin{thm}
   In $QH^*(\Flndot)$, there are quantum relations of the form
       $$
  \sum\limits_{J}(-1)^{|J|} \sigma_{w_J}\sigma_{[1,n_j+d]\backslash J}=0.
$$
\end{thm}
\noindent We will postpone the explanations of the relevant notations to Section 5, and will restate the identity fully in \textbf{Theorem \ref{keyIdentity}}. The proof of the above theorem goes via the complete flag variety $\mathbb{F}\ell_n$ using Peterson's remarkable extension property (see Proposition~\ref{prop:extension}).
 The proof of Theorem \ref{intro:firstchernclass} is closely linked to to the above result.

\subsection*{Remarks for further directions}
Closed string mirror symmetry in full generality at genus zero predicts an isomorphism on the level of Frobenius manifolds. The notion of a \textit{Frobenius manifold} was first introduced by B. Dubrovin in 1990s \cite{Dubr}, while the first   construction of a Frobenius manifold could date back to K. Saito \cite{Saito} in early 1980s in the name of  \textit{flat structures} using his primitive form theory. Mirror symmetry predicts that the Frobenius manifold associated to the Gromov-Witten theory of a Fano manifold $X$ (the \textit{big} quantum cohomology ring of $X$) should be isomorphic to the Frobenius manifold of the mirror Landau-Ginzburg model $(\check X, W)$ associated to an appropriate Saito's primitive form.
This was   indirectly proved for complex Grassmannians in \cite{KiSa, CFKS} by a reduction to the case of projective spaces \cite{Bara}.
The case of quadrics was recently proved in \cite{Hu}, where the verification of Conjecture \ref{conjms} is an important step.
We expect that our Theorem \ref{intro:firstchernclass} will play an important role in studying mirror symmetry $\Flndot$ on such level as well.

 For
 the mirror symmetry on the intermediate level of $D$-modules, the Pl\"ucker coordinate versions of the superpotential of $\mathcal{F}_{\rm Lie}$ play a very important role in proving an explicit injective morphism of $D$-modules for complex Grassmannians and quadrics \cite{MaRi, PRW}. An implicit isomorphism of $D$-modules for minuscule Grassmannians was proved in \cite{LaTe}.
 A proof for an equivariant $D$-module isomorphism for general $G^\vee/P^\vee$ was recently given in \cite{Chow23}. However, the isomorphism therein seems not explicit enough either. Moreover,  verification of the Gauss-Manin connection along $z$-direction was missing, which is an indispensable piece in the mirror symmetry on the level of Frobenius manifolds. We believe that our Theorem \ref{intro:firstchernclass} will be helpful towards getting a better understanding of the $D$-module mirror symmetry for $\Flndot$.

Conjecture \ref{conjms} also appeared in the context of Kontsevich's \textit{homological mirror symmetry} \cite{Kont}, which is one main approach to (open string) mirror symmetry (in addition to another main approach by  Strominger-Yau-Zaslow \cite{SYZ}). For $G^\vee/P^\vee$, homological mirror symmetry was so far only proved for complex Grassmannians $G(n-k, n)$ with $n$ prime \cite{Cast20}, beyond the projective space case covered earlier \cite{Abou}. Here it is important to understand the superpotential $\mathcal{F}_-$ in a Floer theoretical way,  which has only been achieved for very few cases including $Gr(2, n)$ \cite{HKL}. It will be desirable to understand the superpotential more deeply.

Another   closely related direction is about  the Gamma conjecture I  and its underlying conjecture $\mathcal{O}$ proposed by Galkin-Golyshev-Iritani \cite{GGI}.  Here conjecture $\mathcal{O}$ concerns  the eigenvalues  of the aforementioned linear operator
    $\hat c_1|_{\mathbf{q}=\mathbf{1}}$. For flag varieties   $G^\vee/P^\vee$, conjecture $\mathcal{O}$  has already been proved in \cite{ChLi}, while the Gamma conjecture I was only known for very few cases including complex Grassmannians and quadrics.
    One main approach to Gamma conjecture I in \cite{GaIr} relies on a B-side analogy of conjecture $\mathcal{O}$ and a conifold condition. Our Theorem~\ref{firstchernclass}, together with \cite{ChLi}, ensures the  B-side analogy of conjecture $\mathcal{O}$.
    Therefore it will play an important role in the study of the Gamma conjecture I for $\Flndot$ via this approach.

 Finally, we would propose a deeper interaction between mirror symmetry and quantum Schubert calculus for $G^\vee/P^\vee$.  Indeed, for the type $C$ case, some conjectural quantum relations  in the quantum cohomology of a Lagrangian Grassmannian were given in \cite[Conjecture 4.1]{PeRi13}, inspired by Conjecture \ref{conjms}. Even in type $A$, we would ask which quantum relations arise from the natural partial derivatives of the mirror superpotential $\mathcal{F}_-$ via the mirror isomorphism. We also note that some new quantum relations in $QH^*(\Flndot)$ related with a cluster algebra structure were discovered in \cite{HeZh}. It will be interesting to investigate whether these relations could also be revealed using cluster charts in the domain of $\mathcal{F}_-$.

The paper is organized as follows. In Section 2, we introduce the basic notions. In Section 3, we construct two superpotentials $\mathcal{F}_\pm$, and provide the Pl\"ucker coordinate expression of $\mathcal{F}_-$. In Section 4, we prove Theorem \ref{intro:firstchernclass} by assuming Lemma \ref{uinthemiddle} first. Section 5 is devoted to a proof of  Lemma \ref{uinthemiddle} in terms of equivalent  identities on quantum product of Schubert classes. Finally, in the Appendix
we provide a description of $\mathcal{F}_-$ using Young diagrams.

\subsection*{Acknowledgements}
The authors thank   Kentaro Hori, Xiaowen Hu, Elana Kalashnikov, Tony Yue Yu, and Hang Yuan for helpful discussions.  C. Li  is  supported   in part by the National Key Research
and Development Program of China No. 2023YFA100980001 and NSFC Grant   12271529. K. Rietsch is supported by  EPSRC grant EP/V002546/1.

\section{Preliminaries}
We review some background in Lie theory (see e. g. \cite{Borel} for details).

\subsection{Notation}Let $G=GL_n(\mathbb{C})$. 
Let $B_+$ and $B_-$ denote the upper-triangular and lower-triangular Borel subgroups of $G$ with unipotent radicals $U_+$ and $U_-$, respectively. Then $T=B_-\cap B_+$ is the  maximal torus of diagonal matrices in $G$.

Let  $\mathfrak{b}_-, \mathfrak{b}_+, \mathfrak{u}_-, \mathfrak{u}_+, \mathfrak{h}$ be the Lie algebras of $B_-, B_+, U_-, U_+$
and $T$ respectively. 
 Let $\Delta=\{\alpha_1, \cdots, \alpha_{n-1}\}$ be the standard base of simple roots, and $R$ (resp. $R_+$) be the set of (positive) roots. That is, we have the Cartan decomposition $$\mathfrak{gl}(n, \mathbb{C})=\mathfrak{h}\oplus\bigoplus_{\alpha\in R} \mathfrak{g}_\alpha\quad \mbox{with}\quad  \mathfrak{g}_{\alpha_i+\alpha_{i+1}+\cdots+\alpha_{j-1}}=\mathbb{C} E_{ij}\quad\forall 1\leq i<j\leq n,$$
where $E_{i,j}$ is the matrix with entry $1$ in row $i$ and column $j$ and zeros elsewhere.
 We will view elements of $\Delta$ as lying in  the  character group     of $T$, so that 
    $$\alpha_i: T\to \mathbb{C}^*:= {\mathbb{C}\backslash \{0\}}; \quad t=\mbox{diag}(t_1, \cdots, t_n) \mapsto \alpha_i(t)={t_i\over t_{i+1}}.$$

 For any positive integers $k, m$ with $k<m$, we denote the integral interval $[k, m]:=\{k, k+1, \cdots, m\}$, and simply denote  $[m]:=[1, m]$.  Set $I=[n-1]$.   The Weyl group   $W$ of $\mathfrak{gl}(n, \mathbb{C})$ is generated by simple  reflections  $s_i=s_{\alpha_i}$, $i\in I$.
   We will freely identify $W$ with    the Weyl group $N_G(T)/T$ of $G$ as well as the symmetric group $S_n$, by using the isomorphisms
\[\!   
    \begin{array}{lcccccl}
        & \qquad S_n&\overset{\cong}{\longrightarrow}&W &\overset{\cong}{\longrightarrow} &N_G(T)/T,&\\
    \end{array}
    \]
    \[
   \begin{array}{lcccccl} \quad\mbox{where}& (i, i+1)&\mapsto &s_i &\mapsto &\dot s_i T&\mbox{ for}\quad \dot s_i=\exp(E_{i, i+1})\exp(-E_{i+1, i})\exp(E_{i, i+1}).
    \end{array}
    \]
 Moreover, we let  $\ell: W\to \mathbb{Z}_{\geq 0}$ be    the standard length function.
For $w= s_{i_1}\cdots s_{i_m}$ with $m=\ell(w)$,  the element  $\dot w:=\dot s_{i_1}\cdots \dot s_{i_m} \in N_G(T)$ is independent of   the reduced expression chosen.

   We let $P\supseteq B_-$ be a parabolic subgroup of $G$. Set $I_P=\{i\in I\mid \dot s_i\in P\}$ and $I^P=I\setminus I_P$.
   We write
   $$I^{P}=\{n_1, \cdots, n_r\}, $$
   where  $1\leq n_1<n_2<\dots<n_r\leq n-1$.

 Let $W_P$ be the Weyl subgroup of $W$ associated to $P$, and $W^P$ be the set of minimal length coset representatives in $W/W_P$. Precisely, 
   $$W_P=\langle s_i\mid i\in I_P\rangle,\quad W^P=\{u\in W\mid \ell(us_i)>\ell(u),\,\forall i\in I_P\}.$$
 Denote by $w_P$ (resp. $w^P$, $w_0$)  the longest element in $W_P$ (resp. $W^P$, $W$). 

The Langlands dual group $G^\vee$ to $G$ is again $GL(n, \mathbb{C})$, but plays a different role.  Let $B^\vee_{\pm}, P^\vee_{\pm}, T^\vee, \Delta^\vee$  be the Langlands dual  versions of $B_{\pm}, P_{\pm}, T, \Delta$, respectively.  The base $\Delta^\vee$ for $G^\vee$ are canonically identified with the set $\{\alpha_1^\vee, \cdots, \alpha^\vee_{n-1}\}$ of simple coroots for $G$.
 In particular, we have $s_{\alpha_i}=s_{\alpha_i^\vee}$. The Weyl group for $G^\vee$ is again $W$, and $I_{P^\vee}=I_P$ for any parabolic subgroup $P$ of $G$ containing $B_+$ or $B_-$. 
 The deeper relationship between the original group $GL_n(\mathbb C)$ and its Langlands dual group is described by the geometric Satake correspondence \cite{lusztig,Gin:GS,MV}.

\subsection{Langlands dual flag varieties}
A partial flag variety is a quotient of $GL(n, \mathbb{C})$ by a parabolic subgroup on the left or right. We can think of it as parameterizing flags of subspaces (of row vectors) in  {$Hom(\mathbb{C}^n, \mathbb{C})$ }or flags of quotients of the space $\mathbb{C}^n$ (in column vectors), to be precisely described below.
  Since we will focus more on the B-side of mirror symmetry, we will use $G$ there, and leave $G^\vee$ for the A-side of mirror symmetry.

On the B-side, recall that  $P\supseteq B_-$ is the parabolic subgroup of $G$ with $I^{P}=\{n_1, \cdots, n_r\}$.  Denote $n_0:=0$ and $n_{r+1}:=n$, and set
$$a_j:=n_j-n_{j-1},\quad \forall j\in [r+1].$$
Then $P$ consists of block-lower-triangular matrices in $G$ with block-diagonal matrices of the form $\mbox{diag}\{M_{1},M_2, \cdots, M_{r+1}\}$, where
 each $M_j$ is an $a_j\times a_j$ invertible matrix.

  Consider   the partial flag variety $F\ell_{\ndot}=F\ell_{n_1,  \cdots, n_r; n}$ that parameterizes flag of vector subspaces $V_{n_j}$ in $Hom(\mathbb{C}^n,\mathbb{C})$, namely
    $$F\ell_{\ndot}= \{  V_{n_1}   \subset  \dotsb  \subset V_{n_r}  \subset Hom (\mathbb{C}^n,\mathbb{C})\mid
    \dim V_{n_j}=n_j\,, 1\leq j\leq r\}.$$
The Lie group $G$   transitively acts on $F\ell_{\ndot}$ on the right,    inducing an isomorphism
$$\tau_{P}: P\!\setminus\! G\overset{\cong}{\longrightarrow} F\ell_{\ndot}.$$
The isomorphism $\tau_{P}$  sends $Pb$  to the partial flag $V_\bull$ such that $V_{n_j}$ is spanned by the first $n_j$ row vectors of the matrix $b$  for all $j\in [r]$.

On the A-side, we consider the partial flag variety $\Flndot=\mathbb{F}\ell(n; n_r, \cdots, n_1)$ that parameterizes flag of quotients, namely
$$\Flndot= \{ \mathbb{C}^n\twoheadrightarrow \Lambda_{n_r} \twoheadrightarrow\cdots \twoheadrightarrow \Lambda_{n_1} \to 0\mid  \dim \Lambda_{n_j}=n_j,\,\,\forall j\in [r]\}/\sim .$$
Here $\Lambda_\bullet \sim \Lambda'_\bullet$ if and only if there are isomorphisms $\Lambda_{n_j}\to \Lambda_{n_j}'$   making the diagram of the two flags of quotients commutative.
There is a canonical  isomorphism
    $$ G^\vee/P^\vee\overset{\cong}{\longrightarrow} \Flndot, $$
which sends $gP^\vee$ to the class of a flag   $\Lambda_\bullet$ of quotients such that the kernel  of $\mathbb{C}^n\twoheadrightarrow \Lambda_{n_j}$ is the vector subspace spanned by the last $n-n_j$ column vectors of the matrix $g$ for all $j$.

We remark that there are many canonical isomorphisms between group quotients and different parameterizations of flag varieties floating around. Here we are taking the above isomorphisms, to fit the philosophy of Langlands dual as well as the word ``mirror". It does not matter much if different isomorphisms are taken.

\subsection{Open Richardson varieties}
For   $v, w\in W$, with $v\leq w$ with respect to the Bruhat order, we have the \textit{open Richardson subvarieties},
          \begin{align*}
          \mathcal{R}_{v, w}^-&:=(B_-\!\setminus\! B_-\dot v^{-1} B_+)\cap (B_-\!\setminus\! B_-\dot w^{-1} B_-)\subset B_-\!\setminus\! G,\\
             \mathcal{R}^{+}_{v, w}&:=(B_+\!\setminus\! B_+\dot v^{-1} B_-)\cap (B_+\!\setminus\! B_+\dot w^{-1} B_+)\subset B_+\!\setminus\! G.
         \end{align*}
 These are smooth and irreducible of dimension $\ell(w)-\ell(v)$ \cite{KaLu}.Note that the intersections above are empty if $v\not\le w$.
  The  Zariski closures are called a (closed) Richardson varieties and denoted by $\overline{\mathcal{R}}_{v, w}^-$ and $\overline{\mathcal{R}}_{v, w}^+$, respectively.
 We will also consider the following (open) Richardson varieties in  $G/B_-$.
      \begin{align*}
          \mathcal{R}_{v, w}^{R^-}&:=(B_+\dot v B_-/B_-)\cap (B_-\dot w B_-/B_-)\subset G/B_-,\\
            \mathcal{R}_{v, w}^{R^+}&:=(B_-\dot v B_+/B_+)\cap (B_+\dot w B_+/B_+)\subset G/B_+..
       \end{align*}
      Open Richardson  varieties  and their projections  will be our main target spaces on the B-side. 


We use the notation $(P\backslash G)^\circ$  for the top-dimensional projected open Richardson variety inside $P\backslash G$, namely $(P\backslash G)^\circ={\rm pr}_P(\mathcal{R}_{{\rm{id}},w_0 w_P}^-)$ under the projection
${\rm pr}_P:B_-\!\setminus\! G \longrightarrow P\!\setminus\! G$.

\subsubsection{Schubert varieties}
On the A-side, we consider the Bruhat decompositions
     $$X=G^\vee/P^\vee= \bigsqcup_{v\in W^{P}} B_+^\vee\dot{v} P^\vee/P^\vee=\bigsqcup_{w\in W^{P}} B_-^\vee\dot{w} P^\vee/P^\vee.$$
The  Zariski closures of the (opposite) Schubert cells $B_+^\vee\dot v P^\vee/P^\vee$ and $B_-^\vee\dot w P^\vee/P^\vee$,
    $$ X^v=\overline{B_+^\vee\dot v P^\vee/P^\vee},\quad X_w:=\overline{B_-^\vee\dot w P^\vee/P^\vee}$$
 are (opposite) Schubert varieties in $X$. They are of codimension $\ell(v)$ and dimension $\ell(w)$ respectively.
It is well-known that the classical cohomology ring $H^{*}(X, \mathbb{Z})$ has a  $\mathbb{Z}$-basis of Schubert classes $\sigma_v$:
   $$ H^{*}(X, \mathbb{Z})=\bigoplus_{v\in W^{P}} \mathbb{Z}\sigma_v, \mbox{ where } \sigma_v:=\mbox{P.D.}[X^v]\in H^{2\ell(v)}(X, \mathbb{Z}),$$
  and $\mbox{P.D.}[X^v]$ stands for  the Poincar\'e dual of the fundamental homology class of~$X^v$.

  \section{The Pl\"ucker coordinate superpotentials}
 In this section, we will  construct two versions of a superpotential  for $X=G^\vee/P^\vee$ defined on projected open Richardson varieties for $G$.
The first superpotential $\mathcal{F}_+$ is a straightforward extension of a construction for complex Grassmannians given in \cite{MaRi}. It has a natural formula in terms of Pl\"ucker coordinates in the Grassmannian case, as was shown in \cite{MaRi}, but this formula does not generalise well to more general partial flag varieties. The construction of the second superpotential $\mathcal{F}_-$, which has a natural Pl\"ucker coordinate presentation in general, is the main outcome of this section.

\subsection{The superpotential $\mathcal{F}_+$ generalising \cite{MaRi}}\label{s:F+}

Recall we have  the following (open) Richardson variety in $B_+\!\setminus\! G$,
      \begin{align*}
          \mathcal{R}^{+}_{{\rm {id}}, w_0 w_P}&:=(B_+\!\setminus\! B_+ B_-)\cap (B_+\!\setminus\! B_+\dot w_P^{-1}w_0^{-1} B_+)\subset B_+\!\setminus\! G       \end{align*}
and the projection map $\mbox{pr}_{P_+}:B_+\!\setminus G\to P_+\!\setminus G$,  {where $P_+$ is upper-triangular parabolic subgroup with $I^{P_+}=\{n_1,\cdots,n_r\}$}. It is shown in \cite[Lemma 3.1]{KLS} that $\mbox{pr}_{P_+}:\mathcal{R}^{+}_{{\rm id}, w_0w_P}\to(P_{+}\backslash G)^\circ=  \mbox{pr}_{P_+}(\mathcal{R}^{+}_{{\rm id}, w_0w_P})$ is an isomorphism. Moreover, implicitly from \cite{KLS}, the projected open Richardson variety $\mbox{pr}_{P_+}(\mathcal{R}^{+}_{{\rm id}, w_0w_P})$ is the complement of an anti-canonical divisor in $P_+\!\setminus G$.

  As in \cite{MaRi}, we use $GL(n, \mathbb{C})$ as the starting point instead of $PSL(n, \mathbb{C})$ used in \cite{Rie08}, and thus  need to cut down to a codimension one subtorus in $T$. The torus $T$ has an  adjoint action by $W$. Consider the invariant subtorus $$\mathcal T^{W_P}=\{ {t}\in T\mid t_{n}=1,\,\,\dot w {t}\dot w^{-1}= {t}, \forall w\in W_P\}.$$

We define a map
\[
\begin{array}{rccc}
\psi_+:& B_-\cap U_+\mathcal T^{W_P}\dot w_P\dot w_0\inv U_+ &\longrightarrow& (P_+\backslash G)^\circ \times \prod\limits_{i\in I^P} \mathbb{C}^*_q\\
 &b_-=u_1t\dot w_P\dot w_0\inv u_2 &\mapsto & (P_+b_-, \mathbf q(t))
\end{array}
\]
where
 \begin{equation}\label{e:qmap}
\begin{array}{ccl}
{{\mathcal T}}^{W_P}&\overset\sim\longrightarrow & \prod\limits_{i\in I^{P}} \mathbb{C}^*_q\\
t &\mapsto &\mathbf q(t\,):= (\alpha_{n_1}(t\,),\dotsc,\alpha_{n_r}(t\,)).
\end{array}
\end{equation}
It follows from \cite[Section~4]{Rie08} and \cite{KLS} that $\psi_+$ is an isomorphism.

\begin{defn}\label{defF+}
   We define the superpotential $\mathcal{F}_+$ by
\[
  \begin{array}{rcccl}
 \mathcal F_+:    (P_+\backslash G)^\circ  \times \prod\limits_{i\in I^P} \mathbb{C}^*_q &\overset{ \psi_+\inv} \longrightarrow& B_-\cap U_+\mathcal T^{W_P}\dot w_P\dot w_0\inv U_+ &\longrightarrow & \mathbb{C}\\
       (P_+b_-,\mathbf q(t))&\mapsto & b_-=u_1 t\dot w_P\dot w_0\inv u_2 &\mapsto &\sum\limits_{i=1}^{n-1} e_i^*(u_1)+ e_i^*(u_2).
\end{array}
\]
Where $e_k^*: U_+\to \mathbb C$ is the map that sends $u=(u_{ij})$  in $U_+$ to its $(k,k+1)$-entry, namely $e_k^*(u)=u_{k,k+1}$. This is well-defined by  \cite[Lemma 5.2]{Rie08}.
\end{defn}
This definition is a direct translation of the Lie-theoretic superpotential defined in \cite{Rie08} via the isomorphism $\psi_+$.  If $P_+$ is a maximal parabolic, then this definition agrees with the one used to give a Pl\"ucker coordinate superpotential for Grassmannians in \cite{MaRi}. In general, viewing $\mathcal F_+$, as a rational function on the partial flag variety $P_+\!\setminus G$ (depending additionally on parameters $q_i$), there will be a Pl\"ucker coordinate formula also in the partial flag setting. However, it turns out that this construction gives a superpotential  that is not as well-suited for being expressed in terms of Pl\"ucker coordinates as we would like.

\begin{example}\label{ex:FullFlagF+}
Consider the complete flag variety $G^\vee/B^\vee_-$ for $GL_3(\mathbb C)$ and the associated superpotential $\mathcal F_+$. Fix a representative $b_-$ of $P_+b_-$. For a subset $I$ of $\{1,2,3\}$ let $p_{I}$ denote the minor of $b_-$ with column set $I$ and row set the  {last} $| I |$ many rows. Then
\[
\mathcal F_+(P_+b_-, (q_1,q_2))={p_2\over p_1}+{p_{13}\over p_{12}}+ {q_2}{p_1p_{13}\over p_3p_{12}}+ {q_1}{p_2 p_{12}\over p_1 p_{23}}.
\]
This example will be useful for comparison between $\mathcal F_+$ and our alternative version of the superpotential that we construct in Section~\ref{s:F-} (see Example~\ref{ex:FullFlagF-}).
\end{example}

\subsection{Superpotential $\mathcal F_{\operatorname {Lie}}$}\label{s:FLie}  We now give the construction of the original Lie theoretic  superpotential $\mathcal F_{\operatorname{Lie}}$ in a form that is suitable for our applications.  
The following definition is a slight change of conventions on \cite[Definition 6.3]{MaRi} and \cite{Rie08}.

We recall the definition of the torus $\mathcal T^{W_P}$ and the isomorphism from \eqref{e:qmap}, as well as the maps $e_k^*: U_+\to \mathbb C$. We also recall that  $P\supseteq B_-$ is the parabolic subgroup of $G$ with $I^{P}=\{n_1, \cdots, n_r\}$.

 \begin{defn} [The Lie-theoretic superpotential] Let $Z_P:=  B_-\cap U_+\mathcal T^{W_P}\dot w_P\inv \dot w_0  U_+$.
  Define the map $\mathcal F_{\operatorname {Lie}}:Z_P\longrightarrow  \mathbb{C}$ by
$$
  \begin{array}{ccc}
                            b_-=u_1t \dot w_P\inv \dot w_0 u_2 &\mapsto& -\left(\sum\limits_{i=1}^{n-1}e_i^*(u_1)+\sum\limits_{i=1}^{n-1}e_i^*(u_2)\right).
  \end{array}
$$
 \end{defn}
\noindent It is an important fact that the  map $\mathcal F_{\operatorname {Lie}}$ is well-defined even though $u_1$ and $u_2$ are not uniquely determined by $b_-$, see \cite[Lemma 5.2]{Rie08}.

\subsection{The superpotential $\mathcal{F}_-$}\label{s:F-}

In this section we give a non-standard isomorphism from $Z_P$ to the product of the projected open Richardson variety $(P\backslash G)^\circ=\mbox{pr}_P(\mathcal{R}^{-}_{{\rm id}, w_0w_P})$ with $\prod_{i\in I^P} \mathbb{C}^*_q $. The superpotential $\mathcal F_-$ will then be defined as a function on $(P\backslash G)^\circ\times \prod_{i\in I^P} \mathbb{C}^*_q $.

We start by considering the isomorphism
\[
\begin{array}{cccc}
\gamma:&Z_P=   B_-\cap U_+{{\mathcal T}}^{W_P} \dot w_P\inv\dot w_0 U_+ &\longrightarrow & (B_-\cap U_+ \dot w_P\inv\dot w_0U_+)\times\prod_{i\in I^P} \mathbb{C}^*_q \\
   &b_-=u_1t \dot w_P\inv\dot w_0 u_2 &\mapsto& (\b=t\inv b_-,\mathbf q(t)),
\end{array}
\]
where $\mathbf q(t)$ is as in \eqref{e:qmap}.
Here the first factor of the right-hand side may be considered as fiber of $Z_P$ where $t$ equals to the identity element. The $q_{n_i}=\alpha_{n_i}(t)$ can also be obtained directly using minors of $b_-$.

We translate the superpotential $\mathcal F_{\operatorname{Lie}}$ to a function on the right-hand side above, and write down a formula for it for future reference. Namely, we have
$\b=\v\dot w_P^{-1}\dot w_0 \, \u$ for $\b\in B_-\cap U_+ \dot w_P^{-1}\dot w_0U_+$. Then
\begin{equation}\label{e:FLieviagamma}
 \mathcal F_{\operatorname{Lie}}\circ  \gamma^{-1} \,(\b, (q_{i})_{i\in I^P})=-\left( \sum\limits_{i\in I_P}\v_{i,i+1} + \sum\limits_{i\in I^P} q_i \v_{i,i+1}+ \sum\limits_{i=1}^{n-1}\u_{i,i+1}\right).
\end{equation}
The main step in the construction now is to make a choice for $\v$ for which only the quantum terms in the formula above will appear, and the other $\v_{i,i+1}$ vanish.
}
Recall that $\mathcal R^{R^+}_{{\rm id},w_P w_0}=B_- B_+\cap B_+w_P w_0B_+/B_+$.
 We define the variety
\begin{equation}\label{e:Z}
\mathcal Z:=U_+\dot w_P^{-1}\dot w_0 \cap \dot w_P^{-1}\dot w_0 U_- \cap B_-B_+,
\end{equation}
which will play a central role in our constructions.

\begin{lemma}\label{zeta}
Consider the intersection $B_-\cap U_+ \dot w_P\inv  \dot w_0U_+$ and the variety $\mathcal Z$ as defined above. We have the following isomorphisms
\[
\begin{array}{lrccc}
&\zeta_1:&B_-\cap U_+ \dot w_P\inv\dot w_0 U_+ &\longrightarrow& \mathcal R^{R^+}_{{\rm id}, w_P w_0}\\
&&\b&\mapsto &\b B_+,\\
\text{\rm }&&&&\\
&\zeta_2:&\mathcal Z  &\longrightarrow& \mathcal R^{R^+}_{{\rm id},w_P w_0}\\
&& z &\mapsto & zB_+.
\end{array}
\]
We consider the composition $\zeta=\zeta_2\inv\circ\zeta_1$,
\[
\begin{array}{cccc}
\zeta:&B_-\cap U_+ \dot w_P\inv\dot w_0U_+ &\longrightarrow&\mathcal Z\\
&\b&\mapsto &z,
\end{array}
\]
where $z\in \mathcal Z$ is the unique representative with  $zB_+=\b B_+$.
\end{lemma}

\begin{proof}
The map $\zeta_1$ is just the restriction to the fiber over $e\in T^{W_P}$ of the isomorphism $B_-\cap U_+ T^{W_P}\dot w_P\inv\dot w_0 U_+\cong \mathcal R^{R^+}_{{\rm id},w_P w_0}\times T^{W_P}$ from \cite[Section~4.1]{Rie08}. We now show that $\zeta_2$ is an isomorphism.

Note that $\mathcal R^{R^+}_{{\rm id},w_P w_0}$ is the open dense subset of the Bruhat cell $B_+w_P w_0B_+/B_+$ obtained by intersecting with opposite big cell $B_-B_+/B_+$. We have the factorisation $U_+=U_+^P\, U_{+,P}$, where
\begin{eqnarray}\label{e:U+P}
U_+^P&:=U_+\cap \dot w_P\inv U_+\dot w_P,\\
U_{+,P}&:=U_+\cap \dot w_P\inv U_-\dot w_P,\nonumber
\end{eqnarray}
and the map $u\mapsto u\dot w_P\inv\dot w_0 B_+$  from $U_+$ restricts to an isomorphism $U_+^P\to B_+w_P w_0B_+/B_+$. Equivalently, the projection map $g\mapsto gB_+$ restricts to an isomorphism
\begin{equation}\label{e:projiso}
U_+^P\dot w_P^{-1}\dot w_0\overset\sim\longrightarrow B_+w_P w_0B_+/B_+.
\end{equation}
We now rewrite the definition of $\mathcal Z$ as follows,
\begin{eqnarray*}
\mathcal Z&=&U_+\dot w_P^{-1}\dot w_0 \cap \dot w_P^{-1}\dot w_0 U_- \cap B_-B_+  \\
&=&\left(U_+ \cap \dot w_P^{-1} U_+\dot w_P\right)\dot w_P^{-1}\dot w_0 \cap B_-B_+
\\
&=& U_+^P\dot w_P^{-1}\dot w_0 \cap B_-B_+.
\end{eqnarray*}
It follows that \eqref{e:projiso} restricts to an isomorphism $\mathcal Z\to \mathcal R^{R^+}_{{\rm id},w_P w_0}$, and this is precisely the map~$\zeta_2$.
\end{proof}

\begin{lemma}
Recall that  $(P\backslash G)^\circ={\rm pr_P}(\mathcal R^-_{{\rm id},w_0w_P })$ and let $\mathcal Z$ be as defined in \eqref{e:Z}. We have an isomorphism
\[
\begin{array}{cccc}
\pi:&\mathcal Z &\longrightarrow& (P\backslash G)^\circ,
\\
& z &\mapsto & Pz.
\end{array}
\]
\end{lemma}
\begin{proof} Note that
\[\mathcal R^-_{{\rm id},w_0w_P }=\mathcal R^-_{w_P,w_0}\dot w_0=B_-\backslash( B_-\dot w_P\inv B_+\cap B_-\dot w_0\inv B_-)\dot w_0.
\]
Consider $U_+^P$ and $U_{+,P}$ as defined in \eqref{e:U+P} and the factorisation $U_+=U_{+,P}U_+^P$. The Bruhat cell $B_-\backslash B_-\dot w_P\inv B_+$ is isomorphic to $U_+^P$ via the map $u\mapsto B_-\dot w_P\inv u$. It follows that
\begin{equation}\label{e:piiso}
\begin{array}{ccc}
 \dot w_P\inv U_+^P\cap B_-\dot w_0\inv B_-&\to&
 B_-\backslash( B_-\dot w_P\inv B_+\cap B_-\dot w_0\inv B_-)\\
 \dot w_P\inv u&\mapsto& B_-\dot w_P\inv u.
 \end{array}
 \end{equation}
is an isomorphism.
We can rewrite the definition of $\mathcal Z$ as follows,
\begin{eqnarray*}
\mathcal Z&=&U_+\dot w_P^{-1}\dot w_0 \cap \dot w_P^{-1}\dot w_0 U_- \cap B_-B_+  \\
&=&\dot w_P\inv\left(\dot w_P U_+\dot w_P\inv  \cap \dot w_0 U_-\dot w_0\inv \right) \dot w_0\cap (B_-\dot w_0\inv B_-)\dot w_0.
\\
&=&\left (\dot w_P\inv U_+^P \cap B_-\dot w_0\inv B_-\right)\dot w_0.
\end{eqnarray*}
We now translate both sides of the isomorphism from \eqref{e:piiso} by $\dot w_0$ and obtain an isomorphism
\begin{equation*}
\begin{array}{rccl}
 \pi':&\mathcal Z&\to& \mathcal R^-_{{\rm id},w_0w_P },\\
& z&\mapsto& B_-z.
 \end{array}
 \end{equation*}
 The composition of $\pi'$ above with the isomorphism $\mathcal R^-_{{\rm id},w_0w_P }\overset\sim\to (P\backslash G)^\circ$ from \cite{KLS} is the map $\pi:\mathcal Z \to (P\backslash G)^\circ$, which proves that $\pi$ is an isomorphism.
\end{proof}

\begin{defn}[The superpotential $\mathcal F_-$]\label{defF-}
We denote by $\psi_-$ the composition of isomorphisms, where $\mathcal B= B_-\cap U_+ \dot w_P\inv\dot w_0U_+ $,
\begin{equation}\label{e:psi-}
\psi_-:Z_P\overset{\gamma}\longrightarrow\ \mathcal B\times \prod\limits_{i\in I^P} \mathbb{C}^*_q\ \overset{\zeta\times{\rm {id}}}\longrightarrow\ \mathcal Z\times  \prod\limits_{i\in I^P} \mathbb{C}^*_q\  \overset{\pi\times{\rm {id}}}\longrightarrow  \ (P\backslash G)^\circ \times  \prod\limits_{i\in I^P} \mathbb{C}^*_q.
\end{equation}
We now define the superpotential $\mathcal F_-$ by
\[
\mathcal F_-:=\mathcal F_{\operatorname{Lie}}\circ \psi_-\inv:(P\backslash G)^\circ\times  \prod\limits_{i\in I^P} \mathbb{C}^*_q \to \mathbb C.
\]
\end{defn}

\subsection{Notations for $\mathcal F_-$}\label{Notations}
In summary, we have shown above that we may write any element of $(P\backslash G)^\circ$ uniquely as $Pz$ for some $z\in \mathcal Z$. We can then write
 \begin{equation}\label{v}
 z=v\inv \dot w_P\inv\dot w_0
 \end{equation}
 for a unique $v\in U_+$. Next let $\b:=\zeta\inv(z)\in B_-\cap U_+\dot w_P\inv\dot w_0 U_+ $. We can  write
 \begin{equation}\label{u}
 \b=z u\inv=v\inv \dot w_P\inv \dot w_0  u\inv,
 \end{equation}
for a unique $u\in U_+$. Finally, the superpotential $\mathcal F_-$ is computed by
 \begin{equation}\label{e:F-viauv}\mathcal F_-(Pz,q)=\mathcal F_{\operatorname{Lie}}\circ  \gamma^{-1} \,(\b, (q_{i})_{i\in I^P})=  \sum\limits_{i\in I_P}v_{i,i+1} + \sum\limits_{i\in I^P} q_i v_{i,i+1}+ \sum\limits_{i=1}^{n-1}u_{i,i+1} ,
 \end{equation}
 using the description of $\mathcal F_{\operatorname{Lie}}$ in \eqref{e:FLieviagamma}.

\begin{defn}\label{uv}
Given $z\in\mathcal Z$ we will always use the notations above for the related elements $\b,v,u$ with $\b\in B_-\cap U_+\dot w_P\inv\dot w_0 U_+$ and $u,v\in U^+$, satisfying
\begin{equation*}
\begin{array}{rcl}
\b B_+ & =& z B_+,\\
z\quad  &=&v\inv \dot w_P\inv \dot w_0,\\
\b\quad &= & z u\inv = v\inv \dot w_P\inv\dot w_0 u\inv.
\end{array}
\end{equation*}
We will also let $ b :=  \hat b\inv=u \dot w_0\inv \dot w_P v$.
\end{defn}
We can immediately simplify the formula~\eqref{e:F-viauv} using the following lemma.

\begin{lemma}\label{l:vanishing}
Let $z\in\mathcal Z$. For $v\in U_+$ as in Definition~\ref{uv}, we have that $v\in \dot w_P\inv U_+\dot w_P$, and therefore the entry $v_{i,i+1}=0$ for all $i\in I_P$.
\end{lemma}
\begin{remark}
This lemma says that $v\in U_+^P$, in the notation  \eqref{e:U+P} from Lemma~\ref{zeta}.
\end{remark}

\begin{proof}[Proof of Lemma~\ref{l:vanishing}]
By \eqref{e:Z} we have that $z\in \dot w_P^{-1}\dot w_0 U_-$. The lemma follows from this and the fact that $v=\dot w_P\inv \dot w_0z\inv$.
\end{proof}

As a consequence of Lemma~\ref{l:vanishing} we have the formula,
 \begin{equation}\label{e:F-viauv-simple}\mathcal F_-(Pz,\mathbf q)=   \sum\limits_{i\in I^P} q_i v_{i,i+1}+ \sum\limits_{i=1}^{n-1}u_{i,i+1},
 \end{equation}
for $\mathcal F_-$ in terms of $u,v$.

We can now use the formula \eqref{e:F-viauv-simple} as our starting point for studying $\mathcal F_-$. The rest of this section will be devoted to giving a compact description of $\mathcal F_-$ in terms of  Pl\"ucker coordinates.

\subsection{The description of $\mathcal Z$}
We first give a concrete description of our variety \begin{equation}\label{e:Zagain}
\mathcal Z=
U_+\dot w_P^{-1}\dot w_0 \cap \dot w_P^{-1}\dot w_0 U_- \cap B_-B_+.
\end{equation}
Recall that $a_j=n_j-n_{j-1}$ where $I^P=\{n_1,\dotsc, n_r\}$.


\begin{lemma}\label{shape-lemma}
Let $I_{a_j}$ denote the $a_j\times a_j$ identity matrix. If $z\in\mathcal Z$ then $z$ is of the following form,
 \begin{equation}\label{Eq:OlocalCoords}
  \left(\begin{matrix}
      * &   *    &\dotsc &   *   &    *   & I_{a_1}\\
      * &   *    &\dotsc &   *   & (-1)^{n_1}I_{a_2} &   0  \vspace{-4pt} \\
      * &   *    &\iddots&\iddots&   0    &   0  \vspace{-6pt}  \\
  \vdots& \vdots &\iddots&   \iddots   &   \vdots    &  \vdots   \\
      * & (-1)^{n_{r-1}}I_{a_r} &   0   & \dotsb&   0    &   0\\
      (-1)^{n_r} I_{a_{r+1}} & 0 &   0   & \dotsb&   0    &   0
   \end{matrix}\right)_{n\times n}.
 \end{equation}
We write $z=v\inv \dot w_P\inv \dot w_0$ as in Definition~\ref{uv}. Then the matrix  $v\in U_+$ has its non-zero above-diagonal entries given by  
 $$v_{n_j, n_j+1}=(-1)^{n_j+1}z_{n_j,n-n_{j+1}+1}.
 $$
\end{lemma}

\begin{proof}
Given square matrices $A_i$, we let $\operatorname{diag}\{A_1, \cdots, A_m\}$ denote the block-diagonal matrix with diagonal blocks $A_i$. Similarly we write
  $$
   \mbox{antidiag}\{A_1, \cdots, A_m\}:=
  \left(\begin{matrix}
              & &     &       & A_1\\
              & &     &  {A_2} &     \vspace{-4pt} \\
              &&&       &     \vspace{-6pt}  \\
   &\iddots&      &       &     \\
       A_m &      & &      &   \\
   \end{matrix}\right).
 $$
for the anti block-diagonal matrix with blocks $A_i$.

We have that $\mathcal Z\subset U_+\dot w_P^{-1}\dot w_0 \cap \dot w_P^{-1}\dot w_0 U_- $. It follows that $z\in \mathcal Z$ has anti-diagonal blocks according to $\dot w_P\inv\dot w_0$, and all other non-zero entries must lie above and to the left of these blocks.

By direct calculation we have
  $\dot w_0=\mbox{antidiag}\{1,-1,  \cdots, (-1)^{n-1}\}$. Moreover, $\dot w_P$ is block diagonal with $j$-th diagonal block given by  $\operatorname{antidiag}\{1,-1, \cdots, (-1)^{a_j-1}\}$.
Therefore  $\dot w_P\inv=\mbox{diag}\{I^{(1)}_\pm, \cdots, I^{(r+1)}_\pm\}$ where $I^{(j)}_\pm :=\mbox{antidiag}\{(-1)^{a_j-1}\cdots, -1,1\}$. It follows that $\dot w_P\inv \dot w_0$ is an anti-diagonal block matrix with top right-hand block given by $ I_{a_1}$, and $(j+1)$-st block given by $(-1)^{n_{j}}I_{a_{j+1}}$,
\begin{equation}\label{e:wPw0}
\dot w_P\inv \dot w_0=\mbox{antidiag}\{I_{a_{1}},(-1)^{n_1}I_{a_2}, \cdots, (-1)^{n_r}I_{a_{r+1}}\}.
     \end{equation}
Here the overall signs of the blocks follow from the fact that the $n_{j+1}$-st row of $\dot w_P\inv$ is the row vector $\delta^{n_{j}+1}_t$, and the $(n-n_{j})$-th column of $\dot w_0$ is $(-1)^{n_{j}}\delta_{n_{j}+1}^t$.

This finishes the proof that the matrix $z$ has the form indicated in \eqref{Eq:OlocalCoords}. We can now check the formula for $v_{n_j,n_j+1}$.
We apply $v\inv$ to $\delta_{n_j+1}^t$, giving
\begin{equation}\label{e:vinv}
v\inv\cdot\delta_{n_j+1}^t=z\dot w_0\inv\dot w_P\cdot \delta_{n_j+1}^t=(-1)^{n_{j}}z\cdot\delta_{n-n_{j+1}+1}^t
\end{equation}
Here we used the inverse of \eqref{e:wPw0},
\begin{equation}\label{e:w0wP}
\dot w_0\inv \dot w_P=\mbox{antidiag}\{(-1)^{n_r}I_{a_{r+1}}, \cdots, (-1)^{n_1}I_{a_2},I_{a_{1}}\}.
\end{equation}
From \eqref{e:vinv} we get
\[
(v\inv)_{n_j,n_j+1}=(-1)^{n_{j}}z_{n_j,n-n_{j+1}+1}.
\]
The formula now follows, since $(v\inv)_{n_j,n_j+1}=-v_{n_j,n_j+1}$.
\end{proof}

\begin{remark}
The complete description of $\mathcal Z$ is that it consists of those matrices of the form \eqref{Eq:OlocalCoords} for which the upper left-hand corner minors are all non-vanishing. This final condition encodes the intersection with $B_-B_+$ in \eqref{e:Zagain}.
\end{remark}


 \subsection{A  Pl\"ucker coordinate formula for $\mathcal F_-$}
For   positive integers $j\leq  m$, we denote by {$\binom{[m]}{j}$}  the set of
subsets $J$ of $[m]$ of cardinality $j$. We denote the complement $J^c_{(m)}=[m]\setminus J$ simply  as $J^c$ whenever $J\subset [m]$ is well understood.
 We always  write $J, J^c$ as increasing sequences, and define $|J|:=\sum_{i\in J}i$.

We consider the Pl\"ucker embedding
\[
\begin{array}{cccc}
\operatorname{Pl}:& P\backslash G &\longrightarrow & \PP^{\binom{n}{n_1}-1}\times\dotsb\times\PP^{\binom{n}{n_r}-1}\\
&Pg &\mapsto &\left([p_{K_1}(g)]_{K_1\in\binom{[n]}{ n_1}},\cdots, [p_{K_r}(g)]_{K_r\in\binom{[n]}{ n_r}}\right)
\end{array}
\]
where the Pl\"ucker coordinate $p_{K_j}(g)$ of  $Pg$ 
 is  the determinant  of the $n_j\times n_j$ sub-matrix of $g$ with first $n_j$ rows and the columns from $K_j$.

   The next proposition is a combination of Propositions 2.2 and 3.9 in   \cite{LSZ} with respect to the quotient $P\!\setminus\!G$, which was also implicitly contained in \cite{KLS}.

   \begin{prop}\label{propanti}
  The projected open Richardson variety  {\upshape $(P\backslash G)^\circ=\mbox{pr}_P(\mathcal{R}^{-}_{{\rm id}, w_0w_P})$} is  the complement of the anti-canonical divisor $-K_{P\backslash G}$ in $P\!\setminus\!G$, where   {\upshape $$  -K_{P\backslash G} \ =\
    \sum_{i\in I} \mbox{pr}_P(\overline{\mathcal{R}}^{{}\,-}_{s_i, w_0w_P})\ +\
   \sum_{i\in I^P}  \mbox{pr}_P(\overline{\mathcal{R}}^{{}\,-}_{{\rm id},   w_0s_i w_P}).$$}
   \end{prop}

\begin{defn}\label{divisoreqn}
  For any homogeneous polynomial $p$ in Pl\"ucker coordinates, we denote  $\mathcal{V}(p):=\{Pg\in  P\backslash G \mid p([p_{K_1}]_{K_1},\cdots, [p_{K_r}]_{K_r})(Pq)=0\}$ and define
 $$ D_k :=\begin{cases}
    \mathcal{V}(p_{[k]}),&\mbox{ if } k\in \{n_1, \cdots, n_r\},\\
    \mathcal{V}(p_{[n]\setminus [k+1, n-n_1+k]}), & \mbox{ if } 1\leq k<n_1,\\
    \mathcal{V}(p_{[k-n_r+1, k]}), & \mbox{ if } n_r<k\leq n-1,\\
    \mathcal{V}(
     \sum\limits_{J\in\binom{[\min\{k, \hat k\}]}{k-n_j}} (-1)^{|J|} p_{[k]\smallsetminus J}\cdot p_{J\cup [\hat k+1,n]}), &  \begin{array}{l} \!\mbox{if }  n_j<k<n_{j+1} \mbox{ with } j\in[r{-}1]\\
                                                                      \mbox{where}\quad \hat k:= n-n_{j+1}+k-n_j
                                                                                                       \end{array},\\
     \mathcal{V}(p_{[n-n_{k-n+1}, n]}), & \mbox{ if } k\in \{n, \cdots, n-1+r\}.
  \end{cases}
  $$
   \end{defn}

The following proposition from  \cite{LSZ}, provides explicit equations for the irreducible components of the the anti-canonical divisor $-K_{P\!\setminus\!G}$ in terms of Pl\"ucker coordinates.
\begin{prop}[{\cite[Theorem 4.1]{LSZ}}] We have
{\upshape $$D_k=\begin{cases}
    \mbox{pr}_P(\overline{\mathcal{R}}^{ -}_{s_k, w_0w_P}),&\mbox{if } 1\leq k\leq n-1,\\
    \mbox{pr}_P(\overline{\mathcal{R}}^{ -}_{{\rm id}, {w_0s_{n_{k-n+1}} w_P}})&\mbox{if } n\leq k\leq n-1+r,
\end{cases}$$}
and $\sum_{k=1}^{n-1+r}D_k$ is an anti-canonical divisor in $P\backslash G$, denoted as  $-K_{P\backslash G}$.
  \end{prop}
We will give a Pl\"ucker coordinate expansion of $\mathcal F_-$ where each summand has a pole of order $1$ along a (unique) irreducible component $D_k$ of  $-K_{P\backslash G}$, giving rise to a bijection between divisors $D_k$ and summands of $\mathcal F_-$.

 For any $n\times m$ matrix $A\in M_{n\times m}(\mathbb{C})$, we let $\Delta_K^J(A)$ denote the minor with row set $J$ and column set $K$, whenever the sub-matrix is a square matrix.  We will need the following generalization  in \cite{GAE} of the famous  Cramer's rule in linear algebra.

\begin{lemma}[Generalized Cramer's rule] \label{Cramer}
   Let $A\in GL_n(\mathbb{C})$ and $X, Y\in M_{n\times m}(\mathbb{C})$  such that $AX=Y$.  For any $J\in {[n]\choose l}$ and $K\in {[m]\choose l}$ where $l\leq \min\{n, m\}$, we have
    $$\Delta^J_K(X)={\det(A_Y(J, K))\over \det A}$$
    where $A_Y(J, K)$ denotes the matrix  constructed from $A$ by order-preserving replacing the column set $J$ of $A$ by the column set $K$ of $Y$.
\end{lemma}
The special case of $X=A^{-1}$ in the generalized Cramer's rule yields Jacobi Theorem for the minors of an inverse matrix immediately as follows.
\begin{cor}[Jacobi Theorem]\label{Jacobithm}
   Let $A\in GL_n(\mathbb{C})$. For any $J, K\in {[n]\choose l}$ where $l\in [n]$, we write $J^c=[n]\setminus J$ and  $K^c=[n]\setminus K $ in  increasing sequences. We have
     $$\Delta^J_K(A^{-1})={(-1)^{|J|+|K|}\over \det A}\Delta^{K^c}_{J_c}(A).$$
\end{cor}

Using   Lemma~\ref{shape-lemma}, we have the next key proposition.
\begin{prop}\label{uvprop} Let $z\in \mathcal Z$. We define $b,u$ and $v$ as in Definition~\ref{uv}, so that  $b=u \dot w_0\inv\dot w_P v=uz\inv$. Then
   \begin{align*}
        u_{i, i+1}&={\Delta^{[i]}_{[i-1]\cup \{i+1\}}(z) \over \Delta^{[i]}_{[i]}(z)}, \quad \mbox{ for any }\, i\in [n-1];\\
       v_{i, i+1}&=\begin{cases}
        {\Delta^{[n_{j}]}_{\{n-n_{j+1}+1\}\cup [n-n_{j}+1,n]\setminus \{n-n_{j-1}\}}(z) \over \Delta^{[n_{j}]}_{[n-n_{j}+1,n]}(z)},& \mbox{ if } i=n_j \mbox{ with } j\in [r],\\
          0,&\mbox{ otherwise}.
      \end{cases}
   \end{align*}
\end{prop}

\begin{proof}
   Since $u\in U_+$ and $b\in B_-$,  we have $\u:=u^{-1}\in U_+$ and   $\b:=b^{-1}\in B_-$. For $m=n-i$, we let $\u^{(m)}$ (resp. $\b^{(m)})$ be the $n\times m$ matrix obtained by taking the last $m$ columns of $\u$ (resp. $\b$). Since $b=uz^{-1}$, we have   $z\u^{(m)}=\b^{(m)}$.  By using Lemma~\ref{Cramer} and noting   $\b\in B_-$ and $\det z=1$, we have
     \begin{align*}
        \u_{n-m, n-m+1}&=\Delta_{[m]}^{\{n-m\}\cup [n-m+2, n]}(\u^{(m)})\\
         &={\det(z_{\b^{(m)}}(\{n-m\}\cup [n-m+2, n], [m]))}\\
         &={-}({\prod_{j=n-m+1}^n\b_{jj}})\Delta^{[i]}_{[i-1]\cup \{i+1\}}(z),
     \end{align*}
        \begin{align*}
        1&=\Delta_{[m]}^{[n-m+1, n]}(\u^{(m)})
          = (\prod_{j=n-m+1}^n\b_{jj})\Delta^{[i]}_{[i]}(z).
     \end{align*}
  Thus for $i\in [n-1]$, we have
     $$u_{i, i+1}=-\u_{i, i+1}=-{\u_{i, i+1}\over 1}={\Delta^{[i]}_{[i-1]\cup \{i+1\}}(z) \over \Delta^{[i]}_{[i]}(z)}.$$

      By Lemma~\ref{l:vanishing}, we have $v_{i,i+1}=0$ if $i\neq \{n_1, \cdots, n_r\}$; We recall that by Lemma~\ref{shape-lemma}, $z$ is   of the form \eqref{Eq:OlocalCoords} and therefore for $j\in [r]$ we have
      \begin{align*}
        v_{n_j,n_j+1}&=(-1)^{n_{j}+1} z_{n_{j},n-n_{j+1}+1}\\
         &=(-1)^{n_{j}+1} {\Delta^{\{n_{j}\}}_{\{n-n_{j+1}+1\}} (z) }\\
         &= (-1)^{n_{j}+1} (-1)^{a_{j}-1} (-1)^{n_{j-1}}
         {\Delta^{[n_{j}]}_{\{n-n_{j+1}+1\}\cup [n-n_{j}+1,n]\setminus \{n-n_{j-1}\}}(z) \over \Delta^{[n_{j}]}_{[n-n_{j}+1,n]}(z)}\\
        &={\Delta^{[n_{j}]}_{\{n-n_{j+1}+1\}\cup [n-n_{j}+1,n]\setminus \{n-n_{j-1}\}}(z) \over \Delta^{[n_{j}]}_{[n-n_{j}+1,n]}(z)}.
     \end{align*}
\end{proof}

\begin{thm}\label{fminus} Let $(q_{n_1}, \cdots, q_{n_r})$ denote the coordinates of $(\mathbb{C}^*)^r=\prod\limits_{i\in I^P} \mathbb{C}^*_q $.
 The superpotential   $\mathcal F_-: (P\backslash G)^\circ \times (\mathbb{C}^*)^r \longrightarrow \mathbb{C}$ is given by the explicit formula
    \begin{align*}
       \mathcal F_-=& \sum_{i=1}^{n_1-1}{\frac{p_{[i-1]\cup \{i+1\}\cup [n-n_1+i+1, n]}}{ p_{[i]\cup [n-n_1+i+1, n]}}} +\sum_{j=1}^{r-1}\sum_{i=n_j+1}^{n_{j+1}-1} S_i^{(j)}+\sum_{i=n_r+1}^{n-1}{p_{ [i-n_r+1, i+1]\setminus\{i\} }\over p_{ [i-n_r+1, i]}}\\
       &\,\, +\sum_{j=1}^r {p_{[n_j-1]\cup \{n_j+1\}}\over p_{[n_j]}}+ \sum_{j=1}^r q_{n_j}   {p_{\{n-n_{j+1}+1\}\cup [n-n_{j}+1,n]\setminus \{n-n_{j-1}\}}\over p_{[n-n_j+1, n]}},
    \end{align*}
where
 $$S_i^{(j)}={\sum_{J\in\binom{[\min\{i+1, \hat i\}]\smallsetminus \{i\} }{i-n_j}} \epsilon(J)(-1)^{|J|} p_{[i-1]\cup\{i+1\}\smallsetminus J}\cdot p_{J\cup [\hat i+1,n]}\over \sum\nolimits_{J\in\binom{[\min\{i, \hat i\}]}{i-n_j}} (-1)^{|J|} p_{[i]\smallsetminus J}\cdot p_{J\cup [\hat i+1,n]}}$$
  with  $\hat i= n-n_{j+1}+i-n_j$ and $\epsilon(J)=\begin{cases}
     1,&\mbox{if } i+1\notin J,\\
     -1,&\mbox{if } i+1\in J.
  \end{cases}$
\end{thm}
\
\begin{proof}
  We let $\b, b, u, v$ be defined as in Definition~\ref{uv} for given $z\in\mathcal Z$.  By Lemma~\ref{l:vanishing}, we have
  \begin{align*}\mathcal F_-(Pz,\mathbf q)=   \sum\limits_{i\in I^P} q_i v_{i,i+1}+ \sum\limits_{i=1}^{n-1}u_{i,i+1},
 \end{align*}
  Since by Proposition~\ref{uvprop} $u_{i, i+1}$ and $v_{i, i+1}$ are quotients of  minors of $z$, it suffices to interpret those minors  by   the Pl\"ucker coordinates. Recall that for subsets $K\in \binom{[n]}{n_j}$, $p_K(z)=\Delta_K^{[n_j]}(z)$ is the determinant of the submatrix of $z$ with first $n_j$ rows and columns from $K$. Therefore if $i\in \{n_1, \cdots, n_r\}$, this is already done for
 both  $u_{i, i+1}$ and $v_{i, i+1}$,  as given in the last two sums of the expression of $\mathcal F_-$.
  Recall that $z$ is of the form ~\eqref{Eq:OlocalCoords}. For $i<n_1$ then both $\Delta^{[i]}_{[i-1]\cup \{i+1\}}(z)=\Delta^{[n_1]}_{[i-1]\cup \{i+1\}\cup [n-n_1+i+1, n]}(z)$ and $\Delta^{[i]}_{[i]}(z)=\Delta^{[n_1]}_{[i]\cup [n-n_1+i+1, n]}(z)$ hold. Then we have
   $$u_{i, i+1}=\frac{p_{[i-1]\cup \{i+1\}\cup [n-n_1+i+1, n]}}{ p_{[i]\cup [n-n_1+i+1, n]}}.$$
 Next we consider the case $i>n_{r}$. Let  $0_{\mu, \nu}$ denote the zero matrix with $\mu$ rows and $\nu$ columns. then the last $(i-n_r)$ rows   for minors
 $\Delta^{[i]}_{[i-1]\cup \{i+1\}}(z)$ and $\Delta^{[i]}_{[i]}(z)$ are both   given by $\Big((-1)^{n_r}I_{i-n_r}, 0_{i-n_r, n-i+n_r}\Big)_{(i-n_r)\times i}$. The  Laplace expansion on the last $(i-n_r)$ rows leads to the third sum in the expression of $\mathcal F_-$ immediately.

 It remains to discuss $u_{i, i+1}$ for the case  $n_j<i<n_{j+1}$  for some $j\in [r-1]$.
 Denote $k=i-n_j$. The first $i$ rows of $z$  is given by
 \[ 
  \left(\begin{matrix}
      *  &  *   &      *     &\dotsc &   *   &   *  & I_{a_1}\\
      *  &  *   &      *     &\dotsc &   *   & (-1)^{n_1}I_{a_2}&   0   \vspace{-4pt}   \\
   \vdots&\vdots&   \vdots   &\iddots&\iddots&   0  &   0   \\
      *  &  *   &      *     & (-1)^{n_{j-1}}I_{a_j}&   0   &   0  &   0   \\
      *  & (-1)^{n_{j}}I_{k}&0_{k,a_{j+1}-k}&   0   &\dotsb &   0  &   0
   \end{matrix}\right)_{i\times n}\ .
 \] 
Here the first column
block  has size $n-n_{j+1}$.
Using Laplace expansion on the last $k$ rows, we obtain
\[
   \Delta^{[i]}_{[i-1]\cup \{i+1\}}(z) =\ \sum_{J\in\binom{[i-1]\cup \{i+1\}}{k}}
 (-1)^{|[i-n_j+1, i]|+|J|} \epsilon(J)\Delta^{[n_j]}_{[i-1]\cup\{i+1\}\smallsetminus J}(z) \cdot z_J\ ,
\]
where {$z_J$} is the determinant of the submatrix with columns from $J$ and last $k$ rows. Since the last $k$ row is
$\begin{pmatrix}* & (-1)^{n_{k-1}}I_k& 0_{k,a_{j+1}-k}&   0  & \dotsb & 0\end{pmatrix}$.
Its last nonzero column is in position $\hat i:=n{-}n_{j+1}{+}k$, so we may assume that $J\subset[l']\setminus \{i\}$ where $l'=\min\{i+1, \hat i\}$,
as otherwise $z_J=0$ for $J$ occurring in the above sum.
Similarly and more easily, we set $l=\min\{i, \hat i\}$ and have
\[
   \Delta^{[i]}_{[i]}(z) =\ \sum_{\tilde J\in\binom{[l] }{k}}
 (-1)^{|[i-n_j+1, i]|+|\tilde J|}  \Delta^{[n_j]}_{[i ] \smallsetminus \tilde J}(z) \cdot z_{\tilde J}.
\]
 Since  $z$ is of  the specified  form as above,
we have $z_J = \varepsilon p_{J\cup [n-n_{j+1}+k+1,n]}$ and $z_{\tilde J} = \varepsilon p_{\tilde J\cup [n-n_{j+1}+k+1,n]}$, in which $\varepsilon=\pm 1$ depends only on $\{n_1, \cdots, n_j\}$.
Hence, we have $u_{i,i+1}= {\Delta^{[i]}_{[i-1]\cup \{i+1\}}(z) \over \Delta^{[i]}_{[i]}(z)}=S_i^{(j)}$ and the proof is complete.
\end{proof}

\begin{example}\label{ex:FullFlagF-} When $I^P=I$, we have that   $P\backslash G$ is a complete flag variety. In this case, there is no $S_i^{(j)}$-term, and the superpotential $\mathcal F_-$ is simply given by $$ \mathcal F_-=\sum_{i=1}^{n-1}{p_{[i-1]\cup\{i+1\}\over p_{[i]}}}+\sum_{i=1}^{n-1}q_i{p_{[n-i, n]\setminus\{n-i+1\}}\over p_{[n-i+1, n]}}.$$
\end{example}

\begin{example}\label{ex:F-Gr} When $I^P=\{k\}$,  we have that $P\backslash G$  is the complex Grasssmannian $Gr(k, n)$. In this case,  there are no $S_i^{(j)}$-terms, and the superpotential $\mathcal F_-$ is simply given by $$ \mathcal F_-=\sum_{i=1}^{k-1}{\frac{p_{[i-1]\cup \{i+1\}\cup [n-k+i+1, n]}}{ p_{[i]\cup [n-k+i+1, n]}}} +\sum_{i=k+1}^{n-1}{p_{[i-k+1, i+1]\setminus \{i\}}\over p_{[i-k+1,i]}}+
{p_{[k-1]\cup\{k+1\}}\over p_{[k]}}+ q_k{p_{[n-k, n]\setminus\{n-k+1\}}\over p_{[n-k+1, n]}}.$$
\end{example}

\section{Quantum cohomology of partial flag varieties}
The main result of this section will be to  show that the image of the superpotential in the Jacobi ring  agrees with the the first Chern class under the known isomorphism with quantum cohomology (Section~\ref{s:peterson}). This result is Theorem~\ref{firstchernclass}.

On the A-side, we consider the small quantum cohomology ring of $X=G^\vee/P^\vee$, denoted by $QH^*(X)$. It is an associative and commutative algebra over $\mathbb{C}[q_{n_1},\cdots,q_{n_r}]$ with a basis given by the Schubert classes $\sigma_v$, where $q_{n_j}$ are formal variables. $$QH^*(X)=\mathbb{C}[q_{n_1},\cdots,q_{n_r}]\otimes H^*(X,\mathbb{Z})$$
The structure constants are defined through the $3$-point, genus-zero Gromov-Witten invariants of $X$. There is also an enumerative meaning of these constants (see e.g. \cite{FulPand}).
Recall that the number $r$ occurs in the isomorphism $G^\vee/P^\vee\overset{\cong}{\longrightarrow} \Flndot =\mathbb{F}\ell(n; n_r, \cdots, n_1)$.

A permutation $w\in S_n$ is an element in $W^{P}$ if and only if $w(n_j+1)<\cdots <w(n_{j+1})$ for all $0\leq j\leq r$. In particular, if $w=s_{n_j-i+1}\cdots s_{n_j}$, $1\leq j\leq r$, $1\leq i\leq n_j$, then $\sigma_w$ is called a special Schubert class. The following is a special case of the quantum Pieri rule in \cite{Fontanine}.
\begin{prop}[Ciocan-Fontanine]\label{Monkrule}
  Let $w$ be a Grassmannian permutation with $w(1)<\cdots<w(m)$ and $w(m+1)<\cdots<w(n)$  for some $m=n_j$, $1\leq j\leq r$. Then in $QH^*(X)$, we have
\begin{align*}
  \sigma_w \cdot \sigma_{s_{n_j}}=\sum_{\substack {a\leq n_j<b,\\ \ell(wt_{ab})=\ell(w)+1}}\sigma_{wt_{ab}}+ \sum\limits_{\ell(w\tau)=\ell(w)-\ell(\tau)} q_{n_j}\sigma_{w\tau},
\end{align*}
where $t_{ab}$ is the transposition interchanging $a$ and $b$, $\tau:=s_{n_j}\cdot s_{n_j+1}\times\cdots\times s_{n_{j+1}-1}\cdot s_{n_j-1}\cdot s_{n_j-2}\times \cdots \times s_{n_{j-1}+1}$.
\end{prop}
\begin{remark}
  Since $w$ is a Grassmannian permutation, there is at most one quantum term in the expansion of the product $ \sigma_w \cdot \sigma_{s_{n_j}}$. Note that the condition $\ell(w\tau)=\ell(w)-\ell(\tau)$ is equivalent to $w(n_j)>w(n_{j+1}), w(n_j+1)<w(n_{j-1}+1)$.
\end{remark}
\subsection{Peterson isomorphism}\label{s:peterson}
In this section we state three theorems of D. Peterson, of which the proofs may be found in \cite[Section 4]{Rie03}.
\begin{defn}\label{d:Gmi}
  For $1\leq i\leq m<n$, we define a rational function $G_i^m$ on $G/B_-$ as follows:
  $$G_i^m(gB_-):=\frac{\Delta_{[m+1,n]}^{\{m-i+1\}\cup [m+2,n]}(g)}{\Delta_{[m+1,n]}^{[m+1,n]}(g)}.$$
\end{defn}
\begin{defn}[Peterson variety] Let $\mathcal Y$ denote the (type $A$) Peterson variety, which is the projective subvariety of $G/B_-$ cut out by the relation 
\begin{equation}\label{e:PetersonRelation}
g\inv f g\in \mathfrak b_-\oplus \left(\bigoplus\limits_{i\in I}\mathfrak g_{\alpha_i}\right)=\left\{ \begin{pmatrix}
* &*&0&\dots&0 \\
* &\ddots&\ddots&\ddots&\vdots \\
\vdots &\ddots&\ddots&\ddots&0 \\
\vdots&\ddots&\ddots&\ddots& *\\
* &\dots &\dots&*&*
\end{pmatrix}\right\},
\end{equation}
where $g$ represents $gB_-$ and $f$ is the principal nilpotent
\[
f=\begin{pmatrix}
0 &&& \\
1 &0&& \\
 &\ddots&\ddots&\\
 &&1&0
\end{pmatrix}.\]
We set
\[\mathcal Y_{P}:=\mathcal Y\cap B_+\dot w_{P} B_-/B_-,
\] and refer to this intersection as the  Peterson variety associated to the parabolic subgroup $P$.
\end{defn}
\begin{thm}[D. Peterson] \label{petersonA}
  Let $\mathcal{Y}_{P}$ be the Peterson variety associated to the parabolic subgroup $P$. There is a unique isomorphism
  \begin{align*}
    \mathcal{O}(\mathcal{Y}_{P}) &\overset{\sim}\longrightarrow QH^*(G^\vee/P^\vee),\\
    G_i^{n_j} &\mapsto (-1)^{i}\sigma_{s_{n_j-i+1}\cdots s_{n_j}}
  \end{align*}
  where $1\leq j\leq r$, $1\leq i\leq n_j$ and $  G_i^{n_j}$ is as constructed in Definition~\ref{d:Gmi}.
  \end{thm}
  \begin{remark}
     {The isomorphism we are using differs from the one used in \cite{Rie03} by signs.}
  \end{remark}
  \begin{remark}
    The rational function $G_i^m$ is a regular function on the Schubert cell $B_+\dot w_{P} B_-/B_-$ if $m\in I^{P}$.
  \end{remark}
  More generally, for a Grassmannian permutation $w$, with $w(1)<\cdots<w(m)$ and $w(m+1)<\cdots<w(n)$ for some $1<m<n$, we can define a rational function $G_w$ on $G/B_-$ as follows: 
  $$G_w(gB_-):=\frac{\Delta_{[m+1,n]}^{\{w(m+1),\cdots,w(n)\}}(g)}{\Delta_{[m+1,n]}^{[m+1,n]}(g)}.$$
  We use $G^{\{w(m+1),\cdots,w(n)\}}(gB_-):=G_w(gB_-)$ for short. This will not lead to any misunderstanding since our $n$ is fixed throughout the paper.
  Then we have the following result, by \cite[Prop 11.3]{Rie03}.
 {  \begin{prop}
    If $w$ is a Grassmannian permutation with descent $m$ and $m\in I^{P}$, then the isomorphism in Theorem~\ref{petersonA} sends $G_w$ to $(-1)^{\ell(w)}\sigma_w$.
  \end{prop}}

\begin{thm}[D. Peterson]\label{petersonB}
    Let $\mathcal{X}_{P}:=\mathcal{Y}_{P}\cap (B_-\dot w_0 B_-/B_-)$. Then the map in Theorem~\ref{petersonA} induces an isomorphism between $ \mathcal{O}(\mathcal{X}_{P})$ and $QH^*(G^\vee/P^\vee)[q_{n_1}^{-1},\cdots,q_{n_r}^{-1}]$.
\end{thm}

We may also refer to $\mathcal X_P$ as Peterson variety. We recall that $\mathcal X_P$ can be described using Toeplitz matrices using an idea going back to B.~Kostant.
\begin{thm}[D. Peterson]\label{t:Toeplitz}
Consider the following variety of lower-triangular Toeplitz matrices given by
$$X_P:=\left\{ b_-=
  \begin{pmatrix}
    x_1 &  &  &    \\
    x_2 & x_1 &  &  \\
    \vdots& \ddots  &\ddots&  \\
    x_n & \cdots& x_2 & x_1
  \end{pmatrix}\mid b_-\in B_+\dot w_P\dot w_0B_+
\right\}.$$
The map $X_P\to\mathcal X_P$ sending $b_-$ to $b_-\dot w_0 B_-$ is an isomorphism.
\end{thm}

\subsection{Critical points of the superpotential}\label{subseccrit}
We have the following isomorphism $\psi_R$ which is a version of an isomorphism from  \cite[Section 4.1]{Rie08},
       \begin{align*}
      &\psi_R: B_-\cap U_+\mathcal T^{W_P}\dot w_P\inv\dot w_0U_+   \overset{\cong}{\longrightarrow} \mathcal{R}^{R^-}_{w_P, w_0}\times \prod_{i\in I^P} \mathbb{C}^*_q; \\
      & b_-=u_1t\dot w_P\inv\dot w_0u_2\mapsto \psi_R(b_-)=   (b_-\dot w_0B_-, (\alpha_{n_i}(t))_{i=1}^{r}).
     \end{align*}
Define $\mathcal{F}_R$ by 
\[
\mathcal{F}_R:=\mathcal F_{\operatorname {Lie}}\circ \psi_R^{-1}: \mathcal{R}^{R^-}_{w_P, w_0}\times \prod_{i\in I^P} \mathbb{C}^*_q\longrightarrow \mathbb{C}.
\]
We now consider the ring
\begin{equation}\label{e:Jac}
Jac(\mathcal F_R):=\mathcal O\left(\mathcal{R}^{R^-}_{w_P, w_0}\times \prod_{i\in I^P} \mathbb{C}^*_q\right)/(\partial_{\mathcal{R}^{R^-}_{w_P, w_0}} \mathcal{F}_R),
\end{equation}
 where we are taking partial derivatives of $\mathcal F_R$ in the $\mathcal{R}^{R^-}_{w_P, w_0}$ directions, and which we refer to  as the (fiberwise) Jacobi ring of $\mathcal{F}_R$. This ring describes the critical points of $\mathcal{F}_R$ along the fibres of the projection $\operatorname{pr}_2:\mathcal{R}^{R^-}_{w_P, w_0}\times \prod_{i\in I^P} \mathbb{C}^*_q\to \prod_{i\in I^P} \mathbb{C}^*_q$.

The next theorem shows that the Jacobi ring of $\mathcal{F}_R$ is isomorphic to the coordinate ring of the Peterson variety $\mathcal X_P$. Namely, Consider the subvariety of $\mathcal{R}^{R^-}_{w_P, w_0}\times \prod_{i\in I^P} \mathbb{C}^*_q$ cut out by the ideal $(\partial_{\mathcal{R}^{R^-}_{w_P, w_0}} \mathcal{F}_R)$ of partial derivatives of $\mathcal{F}_R$ along $\mathcal{R}^{R^-}_{w_P, w_0}$. We denote the corresponding subvariety in $Z_P=B_-\cap U_+\mathcal T^{W_P}\dot w_P\inv\dot w_0U_+ $ by $Z_P^{crit}$. The theorem stated below is a direct translation of \cite[Theorem 4.1]{Rie08} with our notation.

\begin{thm}\label{critpointthm}
We have that $Z_P^{crit}=X_P$. Moreover, the  subvariety 
\[
\psi_R(Z_P^{crit})\subset \mathcal{R}^{R^-}_{w_P, w_0}\times \prod_{i\in I^P} \mathbb{C}^*_q,
\]
which is defined by the ideal $(\partial_{\mathcal{R}^{R^-}_{w_P, w_0}} \mathcal{F}_R)$  is isomorphic to $\mathcal{X}_{P}$ via the restriction of the first projection $\operatorname{pr}_1:\mathcal{R}^{R^-}_{w_P, w_0}\times \prod_{i\in I^P} \mathbb{C}^*_q \rightarrow \mathcal{R}^{R^-}_{w_P, w_0}$.

\begin{proof} We have the following result proved in \cite{Rie08} that we state for the parabolic subgroup $Q:=\dot w_0 P\dot w_0\inv$. Let us set $\tilde{\mathcal T}^{W_Q}:= \dot w_0\inv {\mathcal T}^{W_P}\dot w_0$. We define
\[
\begin{array}{cccc}
\tilde{\mathcal F}_Q: &\tilde Z_Q=B_-\cap U_+\tilde{\mathcal T}^{W_Q}\dot w_Q\dot w_0\inv U_+& \to&\mathbb C,\\
& \tilde b_-=u_1\tilde t\dot w_Q\dot w_0\inv u_2 &\mapsto & \sum_i e_i^*(u_1)+\sum_i e_i^*(u_2),
\end{array}
\]
and its restrictions
\[
\tilde{\mathcal F}_{Q,\tilde t}: \tilde Z_{Q,\tilde t}=B_-\cap U_+\tilde t\dot w_Q\dot w_0\inv U_+\to\mathbb C.
\]
Then it is shown in  \cite{Rie08} that the critical points of $\tilde{\mathcal F}_{Q,\tilde t}$ lie in $X_Q$. Namely,
\[ \tilde Z_{Q,\tilde t}^{crit}:=\{\tilde b_-\in B_-\cap U_+\tilde t\dot w_Q\dot w_0\inv U_+\mid \partial \tilde{\mathcal F}_{Q,\tilde t}(\tilde b)=0\} =X_Q\cap  \tilde Z_{Q,\tilde t},\]
where $X_Q$ is as in Theorem~\ref{t:Toeplitz}. Moreover the fiberwise critical point variety $ \tilde Z_{Q}^{crit}$ (union over all fibers $\tilde Z^{crit}_{Q,\tilde t}$) agrees with $X_Q$.

We can now compare $\tilde{\mathcal F}_Q$ with our superpotential $\mathcal F_{\operatorname{Lie}}$. We have that
\begin{equation*}
\mathcal F_{\operatorname{Lie}}(b_-=u_1t\dot w_P\inv\dot w_0 u_2)=\tilde{\mathcal F}_Q(b_-\inv = u_2\inv \tilde t\dot w_Q\dot w_0\inv u_1\inv )
\end{equation*}
where $\tilde t=\dot w_0\inv\dot w_P t\inv \dot w_P\inv\dot w_0=\dot w_0\inv t\inv \dot w_0$. Therefore,  $b_-$ is a critical point of the analogous restriction ${\mathcal F}_{\operatorname{Lie},t}$ if and only if $b_-\inv$ is a critical point of $\tilde{\mathcal F}_{Q,\tilde t}$ and we have that the critical point variety $  Z_{P}^{crit}$ of ${\mathcal F}_{\operatorname{Lie}}$ is equal to the inverse of $X_Q$. Finally, the inverse of $X_Q$ is precisely $X_P$ (using that the inverse of a Toeplitz matrix is again Toeplitz). It follows that $\psi_R(Z^{crit}_{\operatorname{Lie}})$ projects to $\mathcal X_P$, thanks to Theorem~\ref{t:Toeplitz}
 \end{proof}

\begin{cor}\label{jacobiringisom}
  The fiberwise Jacobi ring $Jac(\mathcal F_R)$ of the superpotential $\mathcal{F}_R$ is isomorphic to the quantum cohomology ring $QH^*(X)[q_{n_1}^{-1},\cdots,q_{n_r}^{-1}]$.
\end{cor}
\begin{proof} The Jacobi ring is related to $QH^*(X)[q_{n_1}^{-1},\cdots,q_{n_r}^{-1}]$ by
\begin{align*}
Jac(\mathcal F_R)&
\overset{\sim}\longrightarrow \mathcal{O}(\mathcal{X}_{P})\quad\quad\quad\quad\quad\quad\quad\quad\quad\quad\quad\quad (\mbox{Theorem~\ref{critpointthm}})\\
  &\overset{\sim}\longrightarrow QH^*(G^\vee/P^\vee)[q_{n_1}^{-1},\cdots,q_{n_r}^{-1}] \quad\quad\quad (\mbox{Theorem~\ref{petersonB}}).
  \end{align*}
\end{proof}
\end{thm}
We have the following corollary of Theorem~\ref{critpointthm} that we record for use later on. Suppose $\mathbf q\in\prod_{i\in I^P} \mathbb{C}^*_q $. Define
\begin{equation}\label{e:Fq}
{\mathcal F}_{\mathbf q}:B_-\cap U_+\dot w_P\inv\dot w_0U_+\to \mathbb C
\end{equation}
by $\mathcal F_{\mathbf q}(\hat b)=\mathcal F_{\operatorname {Lie}}\circ\gamma\inv(\hat b,\mathbf q)$.
This is precisely the function from \eqref{e:FLieviagamma}, but now with quantum parameters fixed.
\begin{cor}\label{critpoint}
If $\hat b$ is a critical point of $\mathcal F_{\mathbf q}$, then $t \hat b=\gamma\inv(\hat b,\mathbf q)$ is a Toeplitz matrix.
 \end{cor}
\begin{proof}
In the notation from the proof of Theorem~\ref{critpointthm}, we have $\hat b$ is a critical point of $\mathcal F_{\mathbf q}$ if and only if $t\hat b$ is a critical point of $\mathcal F_{\operatorname{Lie},t}$. But the critical points of $\mathcal F_{\operatorname{Lie},t}$ lie in $Z^{crit}_{P}$, which equals to $X_P$ by Theorem~\ref{critpointthm}. Therefore $t\hat b$ is a Toeplitz matrix.
\end{proof}
\subsection{Image of first Chern class}
 The following theorem is the main result in this section, the proof of which is in the end of this subsection.
\begin{thm}\label{firstchernclass}
Let $\theta$ be the isomorphism $Jac(\mathcal{F}_R)\overset{\sim}\longrightarrow QH^*(X)[q_{n_1}^{-1},\cdots,q_{n_r}^{-1}]$ in Corollary~\ref{jacobiringisom}. Let $[\mathcal{F}_R]$ be the class of superpotential $\mathcal{F}_R$ in the Jacobi  ring. Then we have
  $$\theta([\mathcal{F}_R])= {c_1(X)}.$$
\end{thm}
The proof of Theorem~\ref{firstchernclass} will occupy the rest of the paper.

Fix $\mathbf q\in \prod_{i\in I^P}\mathbb C^*_q$ and recall the map $\mathcal F_{\mathbf q}: B_-\cap U_+\dot w_P^{-1}\dot w_0U_+ \to \mathbb C $ defined in~\eqref{e:Fq}.
Let us consider $\hat{b}\in B_-\cap U_+\dot w_P^{-1}\dot w_0U_+ $ and let $z$, $u$, $v$ and $b$ be as in Definition~\ref{uv}. We then have
\begin{equation}\label{e:btwoways}
b=u\dot w_0\inv\dot w_P v=uz\inv.
\end{equation}
Also recall that $b_-=t\b=tb\inv$, with $t\in \mathcal T^{W_P}$ corresponding to $\mathbf q$ via \eqref{e:qmap}.
The following lemmas are key ingredients in the proof of Theorem~\ref{firstchernclass}.
\begin{lemma}\label{vinchernclass} Let $n_j\in I^P$.  Then
  $v_{n_j,n_j+1}=-\frac{t_{n_j+1}}{t_{n_j}} G_1^{n_j}(b_- \dot w_0 B_-)$.
\end{lemma}
\begin{proof}
  Since $u,v\in U_+$ and $\dot w_0\inv \dot w_P=\mbox{antidiag}\{(-1)^{n_r}I_{a_{r+1}}, \cdots, (-1)^{n_1}I_{a_2},I_{a_{1}}\}$, we have
  \begin{align*}
    v_{n_j,n_j+1}&=\frac{\Delta_{[n_j-1]\cup\{n_j+1\}}^{[n_j]}( v)}{\Delta_{[n_j]}^{[n_j]}(v)}\\
    &=\frac{\Delta_{[n_j-1]\cup\{n_j+1\}}^{[n-n_j+1,n]}(\dot w_0\inv\dot w_P v)}{\Delta_{[n_j]}^{[n-n_j+1,n]}(\dot w_0\inv\dot w_P v)}\\
    &=\frac{\Delta_{[n_j-1]\cup\{n_j+1\}}^{[n-n_j+1,n]}(b)}{\Delta_{[n_j]}^{[n-n_j+1,n]}(b)}
  \end{align*}
   Note that for the diagonal entries of $t$ we have that $ t_{n_j+1}=t_{n_j+2}=\cdots= t_{n_{j+1}}$, where $0\leq j\leq r$. Therefore, we have
  \begin{align*}
     v_{n_j,n_j+1}&=\frac{\Delta_{[n_j-1]\cup\{n_j+1\}}^{[n-n_j+1,n]}(b)}{\Delta_{[n_j]}^{[n-n_j+1,n]}(b)}\\
    &= \frac{t_{n_j+1}}{ t_{n_j}}\frac{\Delta_{[n_j-1]\cup\{n_j+1\}}^{[n-n_j+1,n]}(bt\inv)}{\Delta_{[n_j]}^{[n-n_j+1,n]}(bt\inv)}\\
    &= -\frac{ t_{n_j+1}}{t_{n_j}}\frac{\Delta_{[n-n_j]}^{[n_j]\cup [n_j+2,n]}(tb\inv)}{\Delta_{[n-n_j]}^{[n_j+1,n]}(tb\inv)}\\
    &=-\frac{ t_{n_j+1}}{ t_{n_j}} G_1^{n_j}(b_- \dot w_0 B_-).
  \end{align*}
  where in the second to last equality one needs to apply Corollary~\ref{Jacobithm}.
\end{proof}
The situation for $u_{i,i+1}$ is slightly more complicated.
\begin{lemma} {Let $1\leq i< n-n_r$.} Suppose $\hat{b}$ is a critical point of $\mathcal F_{\mathbf q}$, then  $u_{i,i+1}=u_{n-n_r,n-n_r+1}$.
\end{lemma}
\begin{proof}
  Since $u\in U_+$, we have $u_{i,i+1}=\Delta_{[n]\backslash \{i\}}^{[n]\backslash \{i+1\}}(u)$. By Collorary~\ref{Jacobithm}, we have that $\Delta_{[n]\backslash \{i\}}^{[n]\backslash \{i+1\}}(u)=-\Delta_{\{i+1\}}^{\{i\}}(u\inv)$. Applying generalized Cramer's rule, see Lemma~\ref{Cramer}, to the equation $b\inv=zu\inv$, we have $\Delta_{\{i+1\}}^{\{i\}}(u\inv)=\det z_{b\inv}(\{i\},\{i+1\})$.
  Now since $\hat{b}$ is a critical point of $\mathcal F_{\mathbf q}$, we have that $t\b=tb\inv$ is a Toeplitz matrix by Corollary~\ref{critpoint}.
  Note that $ t_{n_r+1}= t_{n_r+2}=\cdots= t_{n}$. Therefore for $1\leq i<n-n_r$ we have
  \begin{align*}
  \det z_{b\inv}(\{i\},\{i+1\})&=\det z_{ t\inv (tb\inv)}(\{i\},\{i+1\})\\
  &= t_{n}\inv\det z_{ (tb\inv)}(\{i\},\{i+1\})\\
   &= t_{n}\inv\det z_{ (tb\inv)}(\{n-n_r\},\{n-n_r+1\})\\
    &=\det z_{ b\inv}(\{n-n_r\},\{n-n_r+1\}).
  \end{align*}
  Therefore, we have shown that $u_{i,i+1}=u_{n-n_r,n-n_r+1}$ whenever $b$ is a critical point of $\mathcal F_{\mathbf q}$ for $1\leq i<n-n_r$.

\end{proof}

\begin{lemma}
  Let $i= n-n_j$ for some $1\leq j\leq r$. Suppose $\hat{b}$ is a critical point of $\mathcal F_{\mathbf q}$, then  $u_{i,i+1}=-G_1^{n_j}(b_- \dot w_0 B_-).$
 \end{lemma}
\begin{proof}
 Since $u,v\in U_+$ and $\dot w_0\inv \dot w_P=\mbox{antidiag}\{(-1)^{n_r}I_{a_{r+1}}, \cdots, (-1)^{n_1}I_{a_2},I_{a_{1}}\}$, we have
  \begin{align*}
    u_{n-n_j,n-n_j+1}&=\frac{\Delta_{[n-n_j+1,n]}^{\{n-n_j\}\cup [n-n_j+2,n]}(u)}{\Delta_{[n-n_j+1,n]}^{[n-n_j+1,n]}(u)}\\
    &=\frac{\Delta_{[n_j]}^{\{n-n_j\}\cup [n-n_j+2,n]}(u\dot w_0\inv\dot w_P)}{\Delta_{[n_j]}^{[n-n_j+1,n]}(u\dot w_0\inv \dot w_P)}\\
    &=\frac{\Delta_{[n_j]}^{\{n-n_j\}\cup [n-n_j+2,n]}(b)}{\Delta_{[n_j]}^{[n-n_j+1,n]}(b)}\\
    &=\frac{\Delta_{[n_j]}^{\{n-n_j\}\cup [n-n_j+2,n]}(bt\inv)}{\Delta_{[n_j]}^{[n-n_j+1,n]}(bt\inv)}.
  \end{align*}
   Now since $\hat{b}$ is a critical point of $\mathcal F_{\mathbf q}$, we have that $t\b=tb\inv$ is a Toeplitz matrix by Corollary~\ref{critpoint}. So that its inverse $bt\inv$ is also a Toeplitz matrix. Therefore, we have
   \begin{align*}
     \frac{\Delta_{[n_j]}^{\{n-n_j\}\cup [n-n_j+2,n]}(bt\inv)}{\Delta_{[n_j]}^{[n-n_j+1,n]}(bt\inv)} &=\frac{\Delta_{[n_j-1]\cup \{n_j+1\}}^{[n-n_j+1,n]}(bt\inv)}{\Delta_{[n_j]}^{[n-n_j+1,n]}(bt\inv)}.
   \end{align*}
   Then it follows as in the proof of Lemma~\ref{vinchernclass} that $ u_{n-n_j,n-n_j+1}=-G_1^{n_j}(b_- \dot w_0 B_-)$.

\end{proof}

\begin{lemma}
Let $n-n_1<i\leq n-1$. Suppose $\b$ is a critical point of $\mathcal F_\mathbf q$, then
  $u_{i,i+1}=-G^{n_1}_1(b_-\dot w_0B_-).$
\end{lemma}
\begin{proof}
By Proposition~\ref{uvprop}, we have   $u_{i, i+1}={\Delta^{[i]}_{[i-1]\cup \{i+1\}}(z) \over \Delta^{[i]}_{[i]}(z)}$. Since $z$ is of the form (\ref{Eq:OlocalCoords}),
 if we set $d=i-(n-n_1)$, then $d<\min\{i,n_1\}$ and we have
 \begin{align*}
  {\Delta^{[i]}_{[i-1]\cup \{i+1\}}(z) \over \Delta^{[i]}_{[i]}(z)}&=-\frac{\Delta_{[i-d]}^{\{d\}\cup\{d+2,i]\}}(z)}{\Delta_{[i-d]}^{[d+1,i]}(z)}\\
  &=-\frac{\Delta_{[i-d]}^{\{d\}\cup\{d+2,i]\}}(b\inv)}{\Delta_{[i-d]}^{[d+1,i]}(b\inv)}.
  \end{align*}
Recall that $b_-=t\b=tb\inv$. Note that $ t_{1}= t_{2}=\cdots= t_{n_1}$ and $1\leq d<d+1\leq n_1$. So we have
 \begin{align*}
   u_{i, i+1}&=-\frac{\Delta_{[i-d]}^{\{d\}\cup\{d+2,i]\}}(b\inv)}{\Delta_{[i-d]}^{[d+1,i]}(b\inv)}\\
   &=-\frac{\Delta^{\{d\}\cup[d+2,i]}_{[n-n_1]}(tb\inv)}{\Delta^{[d+1,i]}_{[n-n_1]}(tb\inv)}\\
&=-\frac{G^{\{d\}\cup [d+2,i]}(b_-\dot w_0 B_-)}{G^{[d+1,i]}(b_- \dot w_0 B_-)}\\
&=-G^{n_1}_1(b_-\dot w_0B_-).
 \end{align*}
The last equality follows from  Proposition~\ref{Monkrule} and the isomorphism in Corollary~\ref{jacobiringisom}.
\end{proof}

\begin{lemma}\label{uinthemiddle}  Suppose $n-n_{j+1}< i< n-n_j$ for some $1\leq j\leq r-1$. Suppose $\b$ is a critical point of $\mathcal F_\mathbf q$, then
  \[
  u_{i,i+1}=-(G^{n_j}_1(b_-\dot w_0 B_-)+G^{n_{j+1}}_1(b_-\dot w_0 B_-)).
  \]
\end{lemma}
We leave the proof of this  lemma to the next section.
Now we are ready to prove Theorem~\ref{firstchernclass} assuming the above lemmas.
\begin{proof}[Proof of Theorem \ref{firstchernclass}]Let $\hat{b}\in B_-\cap U_+\dot w_P^{-1}\dot w_0U_+ $ and let $z$, $u$, $v$ and $b$ be as in Definition~\ref{uv}. Also recall that $b_-=t\b$ and $q_{n_j}=\frac{t_{n_j}}{t_{n_j+1}}$. By Equation~\ref{e:F-viauv-simple}, we have $\mathcal F_R(b_-\dot{w_0}B_-,\mathbf q)=\mathcal F_{\operatorname{Lie}}(b_-)=\mathcal F_-(Pz,\mathbf q)=   \sum\limits_{i\in I^P} q_i v_{i,i+1}+ \sum\limits_{i=1}^{n-1}u_{i,i+1}.$  Suppose that $\b$ is a critical point of $\mathcal F_\mathbf q$. Then by Corollary~\ref{jacobiringisom} and the above lemmas, we have
\begin{align*}
  \theta([\mathcal F_{R}])=\sum_{j=1}^r (n_{j+1}-n_{j-1})\sigma_{s_{n_j}}=c_1(X)\in QH^*(X).
\end{align*}
\end{proof}

\section{Proof of Lemma~\ref{uinthemiddle} and Quantum Schubert calculus}
The goal of this section is to prove  the most difficult of the lemmas, which is Lemma~\ref{uinthemiddle}.

\subsection{Equivalence of Lemma~\ref{uinthemiddle} and an identity in quantum Schubert calculus}

We first prove that Lemma~\ref{uinthemiddle} is equivalent to the identity in quantum Schubert calculus stated in the following theorem.

\begin{defn}\label{d:wJ} Recall that we are considering $n-n_{j+1}<i<n-n_j$ for some $1\leq j\leq r-1$, and let $d:=i-(n-n_{j+1})$.  We set
 \[\Xi:=\left\{J\in \binom{[i]}{d}\mid J\cap [n_j+d+1,n]= \emptyset\right\},\]
 and define Weyl group elements $w_J\in W^P$ for certain $J\in \Xi$ as follows.
 For $J=\{\j_1<\j_2<\cdots <\j_d\}\in \Xi$, let $\{x_1<x_2<\cdots<x_{i-d}\}:=[i]\backslash J$.
\begin{enumerate}
  \item If $n_j\geq d$, then $w_J$ is the following permutation
\begin{align*}
 & \{w(1)<\cdots<w(n_j)\}=\{\j_1<\j_2<\cdots<\j_d<i+1<i+2<\cdots<i+n_j-d\}\\
& \{w(n_j+1)<\cdots<w(n_{j+1})\}=\{x_1<i+n_j-d+1<i+n_j-d+2<\cdots<n-1\}\\
& \{w(n_{j+1}+1)<\cdots<w(n_{j+2})\}=\{x_2<\cdots<x_{n_{j+2}-n_{j+1}}<n\}\\
&\{w(n_{j+2}+1)<\cdots<w(n)\}=\{x_{n_{j+2}-n_{j+1}+1}<\cdots<x_{i-d}\}
\end{align*}
  \item If $n_j <d$, and $x_1< \j_{n_j+1}$ then $w_J$ is defined by
\begin{align*}
  &\{w(1)<\cdots<w(n_j)\}=\{\j_1<\cdots<\j_{n_j}\}\\
  &\{w(n_j+1)<\cdots<w(n_{j+1})\}=\{x_1<\j_{n_j+1}<\cdots<\j_d<i+1<\cdots<n-1\}\\
  & \{w(n_{j+1}+1)<\cdots<w(n_{j+2})\}=\{x_2<\cdots<x_{n_{j+2}-n_{j+1}}<n\}\\
  &\{w(n_{j+2}+1)<\cdots<w(n)\}=\{x_{n_{j+2}-n_{j+1}+1}<\cdots<x_{i-d}\}.
\end{align*}
\end{enumerate}
\end{defn}

\begin{example} Suppose $n_1=2,n_2=4,n=7$. Let $j=1$ and $i=4$ (which indeed satisfies $n-n_{j+1}<i<n-n_j$). Then $d=1$ and 
\[\Xi=\left\{J\in  \binom{[4]}{1}\mid J\cap [4,7]= \emptyset\right\}=\{\{1\},\{2\},\{3\}\}.
\]
Since $n_j=2\ge d=1$ we have a Weyl group element $w_J$ for each $J\in\Xi$. Suppose $J=\{\j_1\}$ and $[4]\setminus J=\{x_1,x_2,x_3\}$. Then the definition of $w_J$ is
\[
w_J(1)=\j_1,\quad w_J(2)=5,\quad w_J(3)=x_1,\quad w_J(4)=6,\quad w_J(5)=x_2,\quad w_J(6)=x_3,\quad w_J(7)= 7. 
\]
For our three choices of $J$ this gives the following three Weyl group elements,
\[
w_{\{1\}}=1526347,\quad w_{\{2\}}=2516347,\quad w_{\{3\}}=3516247,
\]
with descents at $2$ and $4$, so in $W^P$. 
\end{example}
We can now state our theorem. Note that since $w_J\in W^P$, we have an associated Schubert class $\sigma_{w_J}$. If a permutation $w\in W^P$ is Grassmannian, so that it is determined by the values $w(1)<\dotsc<w(m)$ up to $m=n_k$ for some $k$, then we may write $\sigma_{\{w(1),\cdots,w(m)\}}$ for $\sigma_w$. Recall that we set $|J|:=\sum_{i\in J}i$.
\begin{thm}  \label{keyIdentity}
Consider $X=G^\vee/P^\vee$ and fix $i$ such that $n-n_{j+1}<i<n-n_j$ for some $1\le j\le r-1$. Let $d:=i-(n-n_{j+1})$. For each $w_J$ defined above we consider $\sigma_{w_J}\in QH^*(X)$, and we set $\sigma_{w_J}:=0$ for $J\in \Xi$ where $w_J$ is not defined. Then the identity
\begin{align}\label{e:keyIdentity}
  \sum\limits_{J\in \Xi}(-1)^{|J|} \sigma_{w_J}\sigma_{[1,n_j+d]\backslash J}=0
\end{align}
 holds in $QH^*(X)$.
\end{thm}
\begin{remark}
 The quantum product
 $\sigma_{J\cup [i+1, n]}\cdot \sigma_{s_{n_{j+1}}}$ consists of at most one quantum part, say $q_{n_{j+1}} \sigma_{w_J}$. The above definition of $w_J$ is an explicit description of such a class.  
\end{remark}

\begin{lemma}\label{equivalence} The formula for $u_{i,i+1}$ in Lemma~\ref{uinthemiddle} is equivalent to the corresponding identity in Theorem~\ref{keyIdentity}.
\end{lemma}
\begin{proof}

We use determinantal identities to rewrite the $u_{i, i+1}$ in Lemma~\ref{uinthemiddle}. By Proposition \ref{uvprop}, we have   $u_{i, i+1}={\Delta^{[i]}_{[i-1]\cup \{i+1\}}(z) \over \Delta^{[i]}_{[i]}(z)}$. Using Laplace expansion on the first $n-n_{j+1}$ columns, we have
\begin{align*}
  {\Delta^{[i]}_{[i-1]\cup \{i+1\}}(z) \over \Delta^{[i]}_{[i]}(z)}
  &=\frac{\sum\limits_{J\in \binom{[i]}{d}}(-1)^{|J|}\Delta_{[n-n_{j+1}]}^{[i]\backslash J}(z)\cdot\Delta_{[n-n_{j+1}+1,i-1]\cup\{i+1\}}^{J}(z)}{\sum\limits_{J\in \binom{[i]}{d}}(-1)^{|J|}\Delta_{[n-n_{j+1}]}^{[i]\backslash J}(z)\cdot\Delta_{[n-n_{j+1}+1,i]}^{J}(z)}
\end{align*}
where $d:=i-(n-n_{j+1})$. Since $z$ is of the form (\ref{Eq:OlocalCoords}), we have that the determinant $\Delta_{[n-n_{j+1}+1,i-1]\cup\{i+1\}}^{J}(z)$  vanishes if $J\cap (\{n_j+d\}\cup [n_j+d+2,n])\neq \emptyset$,
 and $\Delta_{[n-n_{j+1}+1,i]}^{J}(z)$ vanishes if $J\cap [n_j+d+1,n]\neq \emptyset$.  If we set $$A:=\{J\in \binom{[i]}{d}|J\cap (\{n_j+d\}\cup [n_j+d+2,n])= \emptyset\}$$ and $\Xi=\{J\in \binom{[i]}{d}|J\cap [n_j+d+1,n]= \emptyset\}$  as already defined, then we have
 \begin{align*}
  &\frac{\sum\limits_{J\in \binom{[i]}{d}}(-1)^{|J|}\Delta_{[n-n_{j+1}]}^{[i]\backslash J}(z)\cdot\Delta_{[n-n_{j+1}+1,i-1]\cup\{i+1\}}^{J}(z)}{\sum\limits_{J\in \binom{[i]}{d}}(-1)^{|J|}\Delta_{[n-n_{j+1}]}^{[i]\backslash J}(z)\cdot\Delta_{[n-n_{j+1}+1,i]}^{J}(z)}\\
  =&-\frac{\sum\limits_{J\in A}\eta(J)(-1)^{|J|}\Delta_{[n-n_{j+1}]}^{[i]\backslash J}(z)\cdot\Delta_{[n-n_j]}^{J\cup\{n_j+d\}\cup [n_j+d+2,n]}(z)}{\sum\limits_{J\in \Xi}(-1)^{|J|}\Delta_{[n-n_{j+1}]}^{[i]\backslash J}(z)\cdot\Delta_{[n-n_j]}^{J\cup [n_j+d+1,n]}(z)}
\end{align*}
in which $\eta(J)$ is the function $$\eta(J)=\begin{cases}
     1,&\mbox{if } n_j+d+1\notin J,\\
     -1,&\mbox{if } n_j+d+1\in J.
  \end{cases}$$
Since $b\inv =zu\inv$ and $u\in U_+$, we have
\begin{align*}
  &\frac{\sum\limits_{J\in A}\eta(J)(-1)^{|J|}\Delta_{[n-n_{j+1}]}^{[i]\backslash J}(z)\cdot\Delta_{[n-n_j]}^{J\cup\{n_j+d\}\cup [n_j+d+2,n]}(z)}{\sum\limits_{J\in \Xi}(-1)^{|J|}\Delta_{[n-n_{j+1}]}^{[i]\backslash J}(z)\cdot\Delta_{[n-n_j]}^{J\cup [n_j+d+1,n]}(z)}\\
  =&\frac{\sum\limits_{J\in A}\eta(J)(-1)^{|J|}\Delta_{[n-n_{j+1}]}^{[i]\backslash J}(b\inv)\cdot\Delta_{[n-n_j]}^{J\cup\{n_j+d\}\cup [n_j+d+2,n]}(b\inv)}{\sum\limits_{J\in \Xi}(-1)^{|J|}\Delta_{[n-n_{j+1}]}^{[i]\backslash J}(b\inv)\cdot\Delta_{[n-n_j]}^{J\cup [n_j+d+1,n]}(b\inv)}
\end{align*}
We recall that $b_-=t\b=tb\inv$.
Note that $ t_{n_j+1}= t_{n_j+2}=\cdots= t_{n_{j+1}}$ for all $0\leq j\leq r$. Since $n_j<n_j+d<n_j+d+1\leq n_{j+1}$, we have
\begin{align*}
  u_{i,i+1}&=-\frac{\sum\limits_{J\in A}\eta(J)(-1)^{|J|}\Delta_{[n-n_{j+1}]}^{[i]\backslash J}(b\inv)\cdot\Delta_{[n-n_j]}^{J\cup\{n_j+d\}\cup [n_j+d+2,n]}(b\inv)}{\sum\limits_{J\in \Xi}(-1)^{|J|}\Delta_{[n-n_{j+1}]}^{[i]\backslash J}(b\inv)\cdot\Delta_{[n-n_j]}^{J\cup [n_j+d+1,n]}(b\inv)}\\
&=-\frac{\sum\limits_{J\in A}\eta(J)(-1)^{|J|}\Delta_{[n-n_{j+1}]}^{[i]\backslash J}(tb\inv)\cdot\Delta_{[n-n_j]}^{J\cup\{n_j+d\}\cup [n_j+d+2,n]}(tb\inv)}{\sum\limits_{J\in \Xi}(-1)^{|J|}\Delta_{[n-n_{j+1}]}^{[i]\backslash J}(tb\inv)\cdot\Delta_{[n-n_j]}^{J\cup [n_j+d+1,n]}(tb\inv)}\\
&=-\frac{\sum\limits_{J\in A}\eta(J)(-1)^{|J|}G^{[i]\backslash J}(b_-\dot w_0 B_-)\cdot G^{J\cup\{n_j+d\}\cup [n_j+d+2,n]}(b_-\dot w_0 B_-)}{\sum\limits_{J\in \Xi}(-1)^{|J|} G^{[i]\backslash J}(b_-\dot w_0 B_-)\cdot G^{J\cup [n_j+d+1,n]}(b_-\dot w_0 B_-)}
\end{align*}
Applying the isomorphism in Corollary~\ref{jacobiringisom}, we see that the formula for $u_{i,i+1}$ in Lemma~\ref{uinthemiddle} is equivalent to the following identity in $QH^*(X).$
\begin{align}
&\sum\limits_{J\in \Xi}(-1)^{|J|}(\sigma_{J\cup [i+1,n]}\cdot \sigma_{s_{n_{j+1}}}) \cdot \sigma_{[1,n_j+d]\backslash J}+\sum\limits_{J\in \Xi}(-1)^{|J|}\sigma_{J\cup [i+1,n]}\cdot (\sigma_{[1,n_j+d]\backslash J}\cdot \sigma_{s_{n_j}} )\nonumber \\
&\hskip 5cm=\sum\limits_{J\in A}\eta(J)(-1)^{|J|}\sigma _{J\cup [i+1,n]}\cdot \sigma_{([1,n_j+d-1]\cup \{n_j+d+1\}) \backslash J}.\label{e:QHXidentity}
\end{align}
 It remains to show that this is exactly the identity in Theorem~\ref{keyIdentity}.

Assume that $J=\{\j_1<\j_2<\cdots <\j_d\}$, then by Proposition~\ref{Monkrule},
\begin{align*}
 \sigma_{J\cup [i+1,n]}\cdot \sigma_{s_{n_{j+1}}}=\sum\limits_{1\leq s\leq d}\sigma_{\{\j_1,\cdots, \j_{s-1},\j_s+1, \j_{s+1},\cdots, \j_d,i+1,\cdots,n\}}+q_{n_{j+1}}\sigma_{w_J}.
\end{align*}
Here, we set $\sigma_{\{\j_1,\cdots, \j_{s-1},\j_s+1, \j_{s+1},\cdots, \j_d,i,\cdots,n\}}:=0$ if either $s=d$ and $ \j_d=i$ or $\j_s+1=\j_{s+1}$ holds.
We divide the above sum into two parts as follows
\begin{align*}
 & C_1(J):=\sum\limits_{1\leq s\leq d-1}\sigma_{\{\j_1,\cdots, \j_{s-1},\j_s+1, \j_{s+1},\cdots, \j_d,i+1,\cdots,n\}}\\
  &C_2(J):=\sigma_{\{\j_1,\cdots, \j_{d-1}, \j_{d}+1,i+1,\cdots,n\}}
\end{align*}
Similarly, we have
\begin{align*}
\sigma_{[1,n_j+d]\backslash J}\cdot \sigma_{s_{n_j}}&=\sum\limits_{1\leq j\leq d} \sigma_{[1,n_j+d]\backslash \{\j_1, \cdots, \j_{s-1},\j_s-1,\j_{s+1},\cdots,\j_d\}} +D_2(J)\\
&=D_1(J)+D_2(J)
\end{align*}
Here, $D_1(J)$ is defined as
\begin{align*}
  D_1(J):=\sum\limits_{1\leq j\leq d} \sigma_{[1,n_j+d]\backslash \{\j_1, \cdots, \j_{s-1},\j_s-1,\j_{s+1},\cdots,\j_d\}}
\end{align*}
where we set $\sigma_{[1,n_j+d]\backslash \{\j_1, \cdots, \j_{s-1},\j_s-1,\j_{s+1},\cdots,\j_d\}}:=0$ if either  $s=1$ and $\j_1=1$ or  $\j_s-1=\j_{s-1}$ holds.
And $D_2(J)$ is defined as follows
 $$D_2(J):=\begin{cases}
                \sigma_{[1,n_j+d-1]\cup\{n_j+d+1\}\backslash J}, & \mbox{if } n_j+d\notin J, \\
                0, & \mbox{if } n_j+d\in J.
             \end{cases}$$
Note that since $n_{j+1}>n_j+d$, we have $w(n_{j+1})>w(n_j)$ and therefore there are no quantum terms in the product $\sigma_{[1,n_j+d]\backslash J}\cdot \sigma_{s_{n_j}}$ by the remark after Proposition~\ref{Monkrule}.\\
If $n_j+d\in J$, namely, $\j_d=n_j+d$, then directly from the definition of $A$ and $\Xi$ we have
\begin{align*}
  \sum\limits_{J\in \Xi\atop n_j+d\in J}(-1)^{|J|}C_2(J)  \cdot \sigma_{[1,n_j+d]\backslash J}=\sum\limits_{J\in A\atop n_j+d+1\in J}\eta(J)(-1)^{|J|}\sigma _{J\cup [i+1,n]}\cdot \sigma_{([1,n_j+d-1]\cup \{n_j+d+1\}) \backslash J}.
\end{align*}
If $n_j+d\notin J$, then  we have
\begin{align*}
  \sum\limits_{J\in \Xi \atop n_j+d\notin J}(-1)^{|J|}\sigma_{J\cup [i+1,n]}\cdot D_2(J)= \sum\limits_{J\in A \atop n_j+d+1\notin J}\eta(J)(-1)^{|J|}\sigma _{J\cup [i+1,n]}\cdot \sigma_{([1,n_j+d-1]\cup \{n_j+d+1\}) \backslash J}.
\end{align*}
Moreover, for $J=\{\j_1<\j_2<\cdots <\j_d\}$, we denote $J_s^+:=\{\j_1,\cdots, \j_{s-1},\j_s+1, \j_{s+1},\cdots, \j_d\}$ and $J_s^-:=\{\j_1, \cdots, \j_{s-1},\j_s-1,\j_{s+1},\cdots,\j_d\}$.
Then \begin{align*}
  \sum\limits_{J\in \Xi}(-1)^{|J|}\sigma_{J\cup [i+1,n]}\cdot D_1(J)=\sum\limits_{J\in \Xi}\sum_{s=1}^d(-1)^{|J|}\sigma_{J\cup [i+1,n]}\cdot \sigma_{[1,n_j+d]\backslash J_s^-}.
\end{align*}
Since $\sigma_{[1,n_j+d]\backslash J_s^-}\neq 0$ only if $J_s^-\in J$, in which case $(J_s^-)_s^+=J\in \Xi$, we have
\begin{align*}
 \sum\limits_{J\in \Xi}\sum_{s=1}^d(-1)^{|J|}\sigma_{J\cup [i+1,n]}\cdot \sigma_{[1,n_j+d]\backslash J_s^-}
=&\sum\limits_{J\in \Xi}\sum_{s=1}^{d-1}(-1)^{|J|+1}\sigma_{J_s^+\cup [i+1,n]}\cdot \sigma_{[1,n_j+d]\backslash J}+\\
&\sum\limits_{J\in \Xi \atop n_j+d\notin J}(-1)^{|J|+1}\sigma_{J_d^+\cup [i+1,n]}\cdot \sigma_{[1,n_j+d]\backslash J}.
\end{align*}
Therefore, we have
\begin{multline*}
  \sum\limits_{J\in \Xi }(-1)^{|J|} C_1(J)\cdot\sigma_{[1,n_j+d]\backslash J} + \sum\limits_{J\in \Xi \atop n_j+d\notin J}(-1)^{|J|}C_2(J)  \cdot \sigma_{[1,n_j+d]\backslash J}  +\sum\limits_{J\in \Xi}(-1)^{|J|}\sigma_{J\cup [i+1,n]}\cdot D_1(J)=0.
\end{multline*}
We therefore see that the identity~\eqref{e:QHXidentity} is equivalent to the identity
\begin{align*}
  \sum\limits_{J\in \Xi}(-1)^{|J|} \sigma_{w_J}\sigma_{[1,n_j+d]\backslash J}=0
\end{align*}
in Theorem~\ref{keyIdentity}. Hence, the statement follows.
\end{proof}

The structure of the proof of Lemma~\ref{uinthemiddle} and Theorem~\ref{keyIdentity} is now the following. We will first prove Theorem~\ref{keyIdentity} in the special case where $n_j+n_{j+1}\le n$. It then follows that Lemma~\ref{uinthemiddle} holds whenever $n_j+n_{j+1}\le n$, because of Lemma~\ref{equivalence}. Next, we introduce a symmetry on the domain of the superpotential $\mathcal F_R$ that allows us to deduce the statement of Lemma~\ref{uinthemiddle} for $n_j+n_{j+1}\ge n$ from the one for $n_j+n_{j+1}\le n$. Finally, we obtain Theorem~\ref{keyIdentity} for $n_j+n_{j+1}\le n$, since this proposition and Lemma~\ref{uinthemiddle} are equivalent in every case. This strategy also shows an interaction between mirror symmetry and quantum Schubert calculus. 

\subsection{A special case of Theorem~\ref{keyIdentity}}
In this section we prove Theorem~\ref{keyIdentity} in the case where $n_j+n_{j+1}\le n$. To do this we first prove a  version of the identity \eqref{e:keyIdentity} in the quantum cohomology of the complete flag variety $\mathbb{F}\ell_n$  in the following key lemma.
\begin{lemma}\label{keyidentityinFLn}
Assume that $n_j+n_{j+1}\le n$.  With notations as in Theorem~\ref{keyIdentity}, the following identity, obtained simply by replacing the Schubert classes in \eqref{e:keyIdentity} with corresponding ones for $\mathbb{F}\ell_n$, holds in the quantum cohomology ring $QH^*(\mathbb{F}\ell_n)$.
   \begin{align*}
  \sum\limits_{J\in \Xi}(-1)^{|J|} \sigma^B_{w_J}\sigma^B_{[1,n_j+d]\backslash J}=0.
\end{align*}
 \end{lemma}

To prove the above lemma, we need some preparation. Recall that a permutation $w$ is called \textit{$321-$avoiding} if there does not exist $i<j<k$ such that $w(i)>w(j)>w(k)$. The key observation in the proof of Lemma~\ref{keyidentityinFLn} is the following lemma, that the permutations $w_J$ arising above are all $321-$avoiding. 
\begin{lemma}\label{321permutation}
 The permutations $w_J$ constructed in Definition~\ref{d:wJ} are $321-$avoiding.
\end{lemma}
\begin{proof}
 We will argue by contradiction. Consider the case $n_j\geq d$ first. Suppose that there exists $i_0<j_0<k_0$ satisfying $w_J(i_0)>w_J(j_0)>w_J(k_0)$. Since $i_0<j_0$ and $w_J(i_0)>w_J(j_0)$, we must have $n_j+1\leq j_0$. If we assume that  $j_0\leq n_{j+1}$, then we must have $j_0=n_j+1$ since $w_J(i_0)>w_J(j_0)$. So we have $w_J(j_0)= x_1$. However, this is in contradiction with $j_0<k_0$ and $w_J(j_0)>w_J(k_0)$.
 Therefore we must have $j_0> n_{j+1}$. Since $j_0<k_0$ and $w_J(j_0)>w_J(k_0)$, we have $j_0=n_{j+2}$ and $w_J(j_0)=n$. But this is in contradiction with $w_J(i_0)>w_J(j_0)$. In conclusion, for the case $n_j\geq d$, $w_J$ is a $321-$avoiding permutation. The case $n_j<d$ can be proved similarly.
\end{proof}
\begin{remark} \label{r:321avoidingexample} For example,  the identity in Lemma~\ref{keyidentityinFLn} coming from $\mathbb{F}\ell(7;4,2)$, where $n_j=2,n_{j+1}=4$ and $i=4, d=1$, is
  \begin{align*}
    \sigma^B_{1526347}\cdot \sigma^B_{2314567}-\sigma^B_{2516347}\cdot\sigma^B_{1324567}+\sigma^B_{3516247}\cdot\sigma^B_{1234567}=0,
  \end{align*}
and it involves only $321$-avoiding permutations.
\end{remark}

\begin{defn}
 Let $w\in S_n$ be a permutation, then the code of $w$ is defined as $$c(w)=(c_1,c_2,\cdots, c_n)$$
 where $c_i:=\sharp \{j|i<j, w(j)<w(i)\}$.
\end{defn}
\begin{defn}
  Let $w$ be a $321-$avoiding permutation with code $c(w)=(c_1,\cdots,c_n)$. The flag of the partition is defined as $\phi(w)=\{j_1<j_2<\cdots<j_l\}:=\{j|c_j>0\}$. We define a skew partition $\lambda / \mu$ by embedding it into $\mathbb{Z}\times \mathbb{Z}$ as follows:
  \begin{align*}
    &\lambda_k-\mu_k=c_{j_k}\\
   & \lambda/\mu =\{(k,h): 1\leq k\leq l , k-j_k-c_{j_k}+1\leq h\leq k-j_k\}
  \end{align*}
\end{defn}
\begin{example} We continue with the example from Remark~\ref{r:321avoidingexample}.
\begin{enumerate}
  \item  If $w=1526347$, then the code of $w$ is $c(w)=(0,3,0,2,0,0,0)$. And we have $\{j_1<j_2\}=\{2<4\}$ with $l=2$. Then we have
\begin{align*}
  & k=1,  1-2-3+1\leq h \leq 1-2\\
   & k=2, 2-4-2+1\leq h\leq 2-4
\end{align*}
Therefore, the skew partition is $\ydiagram{3,2}$ with $\lambda=(3,2)$ and $\mu=(0,0)$.
  \item If $w=2516347$, then the code of $w$ is $c(w)=(1,3,0,2,0,0,0)$ with flag $\phi(w)=\{1,2,4\}$. Then
  \begin{align*}
    & k=1, 1-1-1+1\leq h\leq 1-1\\
    &k=2, 2-2-3+1\leq h\leq 2-2\\
    &k=3,3-4-2+1\leq h\leq 3-4
  \end{align*}
  Therefore, the skew partition is $\lambda/\mu=\ydiagram{2+1,3,2}$ with $\lambda=(3,3,2)$ and $\mu=(2,0,0)$.
  \item If $w=3516247$, then the code of $w$ is $c(w)=(2,3,0,2,0,0,0)$ with flag $\phi(w)=\{1,2,4\}$. Then
  \begin{align*}
    &k=1, 1-1-2+1\leq h\leq 1-1\\
    &k=2, 2-2-3+1\leq h\leq 2-2\\
    &k=3, 3-4-2+2+1\leq h\leq 3-4
  \end{align*}
  Therefore, the skew partition is $\lambda/\mu =\ydiagram{1+2,3,2}$ with $\lambda=(3,3,2)$ and $\mu=(1,0,0)$.
\end{enumerate}

\end{example}
In \cite{BJS}, Schubert polynomial $\mathfrak{S}_w$ is explicitly written  down in a determinantal formula for a $321-$avoiding permutation as follows.
\begin{thm}[Corollary 2.3 of \cite{BJS}]\label{321equality}
  Let $w$ be a $321-$avoiding permutation with flag $\phi(w)=(\phi_1<\cdots<\phi_k)$ and skew partition $\lambda /\mu$. Let $X_i=(x_1,x_2,\cdots,x_{i})$. Then we have
  $$\mathfrak{S}_w=\det (h_{\lambda_i-\mu_j-i+j}(X_{\phi_i}))_{1\leq i,j\leq k}$$
  where $h_r(X_i)$ is the complete homogeneous symmetric polynomial  of degree $r$ in   variables $X_i$.
\end{thm}
In \cite{Kirillov}, A.N. Kirillov defines quantum Schubert polynomials and conjectures that the quantum version of the above determinantal formula holds as well,  see~\cite[Conjecture~1]{Kirillov}. We will verify this conjecture in the quantum cohomology ring of the complete flag variety $\mathbb{F}\ell_n$ using the work of \cite{FominGelfandPostnikov}.  {This should have been known to the experts (see e.g. \cite[formula (6)]{ChKa}).  Since we are not aware of this formula appearing in form of a theorem, we state it here as Theorem~\ref{321equalityquantum} and provide a detailed argument for completeness.}

\begin{defn}
  Let $G_k$ be the matrix
  $$\left(
    \begin{array}{ccccc}
      x_1 & q_1 & 0 & \cdots & 0 \\
     -1&x_2 &q_2 & \cdots & 0 \\
      0 & -1 & x_3 & \cdots & 0 \\
      \vdots & \vdots & \vdots &\ddots &\vdots \\
      0 &0& 0 & \cdots&x_k \\
    \end{array}
  \right)$$
  The quantum elementary polynomial $E_i^k$ is defined by the following formula
  $$\det(1+\lambda G_k)=\sum_{i=0}^k E_i^k \lambda^i$$
  And we set $E_i^k=0$ if $i<0$ or $i>k$.
\end{defn}
By setting $q_1=q_2=\cdots=q_{k-1}=0$, $E_i^k$ recovers the ordinary elementary symmetric polynomial $e_i^k=e_i^k(x_1, \cdots, x_k)$. Let $e_{i_1\cdots i_m}:=e_{i_1}^1\cdots e_{i_m}^m$ be standard elementary monomial. The following lemma is a classical result, and can be found in \cite{Macdonald}.
\begin{lemma}\label{schubertpoly}
Let $I_n$ be the ideal in  $\mathbb{Z}[x_1,\cdots,x_n]$ generated by $e_1^n,\cdots,e_n^n$, then each of the following forms a $\mathbb{Z}-$basis in $\mathbb{Z}[x_1,\cdots,x_n]/I_n$:
\begin{enumerate}
  \item the standard elementary monomials $e_{i_1\cdots i_{n-1}}$, with $0\leq i_k\leq k$;
  \item the Schubert polynomials $\mathfrak{S}_w$ for $w\in S_n$.
\end{enumerate}
Moreover, each of these families spans the same vector space $L_n\subset \mathbb{Z}[x_1,\cdots,x_n]$ which is complementary to $I_n$.
\end{lemma}
Therefore, any Schubert polynomial $\mathfrak{S}_w$ is uniquely a linear combination of standard elementary monomials with integer coefficients.
In \cite{FominGelfandPostnikov}, the quantum Schubert polynomial is defined as the linear combination of the quantum elementary monomials $E_{i_1\cdots i_m}:=E_{i_1}^1\cdots E_{i_m}^m$ with the same coefficients. Namely, we have
\begin{defn}\label{defnofquantumSchubert}
  The quantum Schubert polynomial $\mathfrak{S}_w^q$ for a permutation $w\in S_n$ is defined as
  $$\mathfrak{S}_w^q=\sum\alpha_{i_1\cdots i_{n-1}}E_{i_1\cdots i_{n-1}}$$
  where the coefficients $\alpha_{i_1\cdots i_{n-1}}$ are the same as the coefficients found in the classical expansion $\mathfrak{S}_w=\sum\alpha_{i_1\cdots i_{n-1}}e_{i_1\cdots i_{n-1}}$.
\end{defn}
We recall the quantum analogue of Lemma~\ref{schubertpoly} proved in \cite{FominGelfandPostnikov}.
\begin{lemma}\label{schubertpolyquantum}
Let $I_n^q$ be the ideal in  $\mathbb{Z}[q_1,\cdots,q_{n-1}][x_1,\cdots,x_n]$ generated by $E_1^n,\cdots,E_n^n$, then each of the following determines a $\mathbb{Z}[q]-$basis in $\mathbb{Z}[q,x]/I_n^q$:
\begin{enumerate}
  \item the quantum standard elementary monomials $E_{i_1\cdots i_{n-1}}$, with $0\leq i_k\leq k$;
  \item the quantum Schubert polynomials $\mathfrak{S}_w^q$ for $w\in S_n$.
\end{enumerate}
Moreover, each of these families spans the same vector space $L_n^q\subset \mathbb{Z}[q,x]$ which is complementary to $I_n^q$.
\end{lemma}
One of the main results in \cite{FominGelfandPostnikov} is the following
\begin{thm}[Theorem 1.2 of \cite{FominGelfandPostnikov}]
 The map
 \[\pi: \mathbb{Z}[q_1,\cdots, q_{n-1}][x_1,\cdots,x_n]\longrightarrow QH^*(\mathbb{F}\ell_n)
 \]
 sending $x_1+\cdots+x_i$ to $\sigma^B_{s_i}\in QH^*(\mathbb{F}\ell_n)$ is a surjective ring homomorphism with kernel $I_n^q$ generated by $E_1^n,\cdots,E_n^n$. Under the induced isomorphism $\mathbb{Z}[q,x]/I_n^q \cong QH^*(\mathbb{F}\ell_n)$, the coset of the quantum Schubert polynomial $\mathfrak{S}_w^q$ is sent to the corresponding quantum Schubert class $\sigma^B_w$.
 \end{thm}
 Now we are ready to prove the quantum version of the determinantal formula for a $321-$avoiding permutation.
 \begin{defn}
  We call $H_l^k:=\det(E_{j-i+1}^{k+l-1})_{1\leq i,j\leq l}$ the quantum complete homogeneous polynomial in $k$ variables of degree $l$.  Set $H_{i_1,\cdots,i_{n-1}}:=H^1_{i_1}\cdots H^{n-1}_{i_{n-1}}$.
 \end{defn}
\begin{remark}
    $H_{i_k}^k \in I_n^q$ if $i_k> n-k$.
   \end{remark}

 \begin{thm}\label{321equalityquantum}
    Let $w$ be a $321-$avoiding permutation with flag $\phi(w)=(\phi_1<\cdots<\phi_k)$ and skew partition $\lambda /\mu$. Let $X_i=(x_1,x_2,\cdots,x_{i})$. Then in $\mathbb{Z}[q,x]/I_n^q$ we have
  $$\mathfrak{S}_w^q=\det (H_{\lambda_i-\mu_j-i+j}(X_{\phi_i}))_{1\leq i,j\leq k}.$$
 \end{thm}
\begin{proof}
   We consider the involution $\omega$ of $\mathbb{Z}[q_1,\cdots,q_{n-1}][x_1,\cdots,x_n]$ defined by $\omega(x_k)=-x_{n+1-k}$ and $\omega(q_k)=q_{n-k}$, for $1\leq k\leq n$. According to \cite{FominGelfandPostnikov}, $I_n^q$ is an invariant subspace for the involution $\omega$.  Therefore    $\omega$ induces an automorphism on $\mathbb{Z}[q,x]/I_n^q$. Moreover, we have
   \begin{align*}
     \omega(E_{i_1\cdots i_{n-1}})=H_{i_{n-1}\cdots i_1};\quad \omega(H_{i_1\cdots i_{n-1}})=E_{i_{n-1}\cdots i_1}; \quad \omega(\mathfrak{S}_w^q)=\mathfrak{S}_{w_0ww_0}^q.
   \end{align*}
Therefore it suffices to show
 $$\mathfrak{S}_{w_0ww_0}^q=\det (E_{\lambda_i-\mu_j-i+j}(X_{n-\phi_i}))_{1\leq i,j\leq k}$$
Note that the right hand side of the equality is a linear combination of quantum standard elementary monomials by the definition of determinants. Then by Lemma~\ref{schubertpolyquantum} it suffices to show that the coefficient of any standard elementary monomial $E_{i_1,\cdots,i_{n-1}}$ with $0\leq i_k\leq k$ on the right hand side is the same as in the definition of the quantum Schubert polynomial. But by Definition~\ref{defnofquantumSchubert}, it suffices to show this in the classical case.
However, by applying involution to Theorem~\ref{321equality}, we have the following equality in $\mathbb{Z}[x_1,\cdots,x_n]/I_n$
$$\mathfrak{S}_{w_0ww_0}=\det (e_{\lambda_i-\mu_j-i+j}(X_{n-\phi_i}))_{1\leq i,j\leq k}$$
Since the right hand side is a linear combination of $e_{i_1,\cdots,i_{n-1}}$ and the standard elementary monomials $e_{i_1,\cdots,i_{n-1}}$ with $0\leq i_k\leq k$ span a vector space complementary to $I_n$, by discarding the other monomials in the expansion of the determinant, we get the formula
$\mathfrak{S}_{w_0ww_0}=\sum_{0\leq i_k\leq k} \alpha_{i_1\cdots i_{n-1}}e_{i_1\cdots i_{n-1}}$ as wanted.
 \end{proof}
 \begin{remark}
In the proof we used the involution $\omega$, therefore we are only able to prove the identity  $\mathfrak{S}_w^q=\det (H_{\lambda_i-\mu_j-i+j}(X_{\phi_i}))_{1\leq i,j\leq k}$ in the quotient ring $\mathbb{Z}[q,x]/I_n^q$. However, the original conjectural identity in \cite{Kirillov} is stated in the ring $\mathbb{Z}[q,x]$.
 \end{remark}
 We now use this theorem 
 to prove Lemma~\ref{keyidentityinFLn}.

\begin{proof}[Proof of Lemma~\ref{keyidentityinFLn}]
 Using the isomorphism $\mathbb{Z}[q,x]/I_n^q \cong QH^*(\mathbb{F}\ell_n)$, we may identify $\mathfrak{S}_w^q$ with $\sigma^B_w$, and treat $H_r(X_i)$ as an element in $QH^*(\mathbb{F}\ell_n)$. Also we use $\times$ for the multiplication. Since $w_J$ is a $321-$avoiding permutation by Lemma~\ref{321permutation}, we are able to apply Theorem~\ref{321equalityquantum}. The proof is divided into two cases: $n_j\geq d$ and $n_j< d$.
 \begin{enumerate}
   \item Consider the case $n_j\geq d$ first. Then for $J=\{\j_1<\cdots <\j_d\}\in \Xi=\{J\in \binom{[i]}{d}|J\cap [n_j+d+1,n]= \emptyset\}$, let $\{x_1<x_2<\cdots<x_{i-d}\}:=[i]\backslash J$, $w_J$ is the following permutation
 \begin{align*}
 & \{w(1)<\cdots<w(n_j)\}=\{\j_1<\j_2<\cdots<\j_d<i+1<i+2<\cdots<i+n_j-d\}\\
& \{w(n_j+1)<\cdots<w(n_{j+1})\}=\{x_1<i+n_j-d+1<i+n_j-d+2<\cdots<n-1\}\\
& \{w(n_{j+1}+1)<\cdots<w(n_{j+2})\}=\{x_2<\cdots<x_{n_{j+2}-n_{j+1}}<n\}\\
&\{w(n_{j+2}+1)<\cdots<w(n)\}=\{x_{n_{j+2}-n_{j+1}+1}<\cdots<x_{i-d}\}.
\end{align*}
The code of $w_J$ is $c(w_J)=(\j_1-1,\j_2-2,\cdots,\j_d-d,i-d,i-d,\cdots,i-d, {0}, i-d-1,\cdots,i-d-1,0,\cdots,0, {n-n_{j+2}},0,\cdots,0)$ with flag $\phi(w_J)=(1,2,\cdots,n_j,n_j+2,n_j+3,\cdots,,n_{j+1},n_{j+2})$. Then it determines a skew partition $\lambda/\mu$, where
\begin{align*}
  &\lambda=(i-d,\cdots,i-d,i-d-1,\cdots,i-d-1,n-n_{j+2}) \mbox{ with } n_j \mbox{ many }  i-d \\
  &\quad \quad \mbox{and } n_{j+1}-n_j-1 \mbox{ many } i-d-1;\\
   &\mu=(i-d-(\j_1-1),i-d-(\j_2-2),\cdots,i-d-(\j_d-d),0,\cdots,0).
\end{align*}

Then by Theorem~\ref{321equalityquantum}, we have $\sigma^B_{w_J}=\det (H_{\lambda_r-\mu_s-r+s}(X_{\phi_r}))_{1\leq r,s\leq n_{j+1}}$. {Here we do assume $n_{j+2}<n$, the case $n_{j+2}=n$ can be dealt with similarly.} We are going to use Laplace expansion on the first $d$ columns of this determinant. Let $R=(r_1<\cdots<r_d)\in \binom{[n_{j+1}]}{d}$ be row index for the expansion, and denote $M_R$ for the cofactor (with sign) obtained by removing the first $d$ columns and rows indexed by $R$. Then  Laplace expansion says that $$\det (H_{\lambda_r-\mu_s-r+s}(X_{\phi_r}))_{1\leq r,s\leq n_{j+1}}=\sum \limits_R M_R\times\det(H_{\lambda_r-\mu_s-r+s}(X_{\phi_r}))_{1\leq s\leq d, r\in R}.$$
We observe that $M_R$ is independent of $J$ since it  involves only the last $n_{j+1}-d$ columns of $\det (H_{\lambda_r-\mu_s-r+s}(X_{\phi_r}))_{1\leq r,s\leq n_{j+1}}$ and only the first $d$ items of $\mu$ depend on $J$. Therefore, we have
\begin{align*}
  & \sum\limits_{J\in \Xi}(-1)^{|J|} \sigma^B_{w_J}\sigma^B_{[1,n_j+d]\backslash J}\\
   =&\sum\limits_{J\in \Xi} \sum \limits_R  (-1)^{|J|}M_R\times\det(H_{\lambda_r-\mu_s-r+s}(X_{\phi_r}))_{1\leq s\leq d, r\in R}\times \sigma^B_{[1,n_j+d]\backslash J}\\
   =&\sum \limits_R M_R\sum\limits_{J\in \Xi}  (-1)^{|J|}\det(H_{\lambda_r-\mu_s-r+s}(X_{\phi_r}))_{1\leq s\leq d, r\in R}\times \sigma^B_{[1,n_j+d]\backslash J}.
\end{align*}
The Schubert class $\sigma^B_{[1,n_j+d]\backslash J}$ is indexed by a Grassmannian permutation, which in particular is a $321-$avoiding permutation.
Let $\alpha=(\alpha_1,\cdots,\alpha_{n_j})$ be the corresponding partition such that $J\cup \{\alpha_1+n_j,\cdots,\alpha_{n_j}+1\}=[1,n_j+d]$. Then we have
\begin{align*}
  \sigma^B_{[1,n_j+d]\backslash J}=\det (H_{\alpha_a-a+b}(X_{n_j+1-b}))_{1\leq a,b\leq n_j}.
\end{align*}
We will construct an $(n_j+d)\times (n_j+d)$ matrix $A_R$.  We define the first $d$ row vectors of $A_R$ to be $$(H_{\lambda_r-r-i+d+1}(X_{\phi_r}), H_{\lambda_r-r-i+d+2}(X_{\phi_r}),\cdots,H_{\lambda_r-r-i+d+n_j+d}(X_{\phi_r}))$$
where $r$ runs through $R=(r_1< \cdots <r_d)$. And we define the last $n_j$ row vectors of $A_R$ to be
$$(H_{1-n_j-1+b}(X_{n_j+1-b}),H_{2-n_j-1+b}(X_{n_j+1-b}),\cdots,H_{n_j+d-n_j-1+b}(X_{n_j+1-b}))$$
where $b$ run through $[1,n_j]$.

 Next we show that $\det A_R=0$. We will prove this by showing that either $A_R$ contains two identical row vectors or $A_R$ contains a zero row vector. We observe that $\lambda_r-r-i+d+\phi_r=0$, therefore, in order to show that $A_R$ contains two identical row vectors it suffices to prove that $\phi_r=n_j+1-b$ for some $r\in R$ and $b\in[1,n_j]$, namely, $R\cap [1,n_j]\neq \emptyset$. Now suppose we have the opposite, namely $R\cap [1,n_j]=\emptyset$. Then we have $r_1\geq n_j+1$ and $r_d\geq n_j+d$. So we have $\lambda_{r_d}-r_d-i+d=-\phi_{r_d}<-(n_j+d)$. Therefore the $d^{th}$ row of $A_R$ is a zero vector since $H_m(X):=0$ for $m<0$. In conclusion, we have $\det A_R=0$.

  Note that $\lambda_r$ is independent of $J\in \Xi$ and $\mu_s=i-d-(\j_s-s)$, so $\lambda_r-\mu_s-r+s=\lambda_r-r-i+d+\j_s$. Also note that $\alpha_a-a+b=\alpha_a+(n_j+1-a)-n_j-1+b$, while $\alpha_a+(n_i+1-a)$ lies in the complement of $J\subseteq [1,n_j+d]$.   { By our assumption that $n_j+n_{j+1}\le n$, we have    $i\geq n_j+d=n_j+i-(n-n_{j+1})$}. Therefore we have $\Xi=\binom{[n_j+d]}{d}$. Then by taking the Laplace expansion on the first $d$ rows of $A_R$, we see that
$$\det A_R=\sum\limits_{J\in \Xi}  (-1)^{|J|}\det(H_{\lambda_r-\mu_s-r+s}(X_{\phi_r}))_{1\leq s\leq d, r\in R}\times \sigma^B_{[1,n_j+d]\backslash J}.$$
Therefore, under the assumption $n_j+n_{j+1}\leq n$, we have
\begin{align*}
  & \sum\limits_{J\in \Xi}(-1)^{|J|} \sigma^B_{w_J}\sigma^B_{[1,n_j+d]\backslash J}\\
      =&\sum \limits_R M_R\sum\limits_{J\in \Xi}  (-1)^{|J|}\det(H_{\lambda_r-\mu_s-r+s}(X_{\phi_r}))_{1\leq s\leq d, r\in R}\times \sigma^B_{[1,n_j+d]\backslash J}\\
      =&\sum \limits_R M_R \det A_R\\
      =&0.
\end{align*}
\item For the case $n_j<d$, the proof is similar.
   Let $J=\{\j_1<\cdots<\j_d\}\in \Xi=\{J\in \binom{[i]}{d}|J\cap [n_j+d+1,n]=\emptyset\}$. Let $\{x_1<\cdots<x_{i-d}\}=[i]\backslash J$. We consider those $J$ with $x_1<\j_{n_j+1}$ only. Then $w_J$ is defined as
   \begin{align*}
  &\{w(1)<\cdots<w(n_j)\}=\{\j_1<\cdots<\j_{n_j}\}\\
  &\{w(n_j+1)<\cdots<w(n_{j+1})\}=\{x_1<\j_{n_j+1}<\cdots<\j_d<i+1<\cdots<n-1\}\\
  & \{w(n_{j+1}+1)<\cdots<w(n_{j+2})\}=\{x_2<\cdots<x_{n_{j+2}-n_{j+1}}<n\}\\
  &\{w(n_{j+2}+1)<\cdots<w(n)\}=\{x_{n_{j+2}-n_{j+1}+1}<\cdots<x_{i-d}\}.
\end{align*}
The code of $w_J$ is $c(w_J)=(\j_1-1,\j_2-2,\cdots,\j_{n_j}-n_j, {0},\j_{n_j+1}-n_j-2,\j_{n_j+2}-n_j-3,\cdots,\j_{d}-d-1,i-d-1,i-d-1,\cdots,i-d-1,0,\cdots,0, {n-n_{j+2}},0,\cdots,0)$
with flag $\phi(w_J)=(1,2,\cdots,n_j,n_j+2,n_j+3,\cdots,n_{j+1},n_{j+2})$. Then it determines a skew partition $\lambda/\mu$, where
\begin{align*}
   &\lambda=(i-d,\cdots,i-d,i-d-1,\cdots,i-d-1,n-n_{j+2}) \mbox{ with } n_j \mbox{ many }  i-d \\
  &\quad \quad \mbox{and } n_{j+1}-n_j-1 \mbox{ many } i-d-1;\\
   &\mu=(i-d-(\j_1-1),i-d-(\j_2-2),\cdots,i-d-(\j_{d}-d),0,\cdots,0).
\end{align*}
We notice that the flag $\phi(w_J)$ and the skew partition $\lambda/\mu$ are the same as the case $n_j\geq d$. Therefore, the rest of the proof is similar to the case $n_j\geq d$.
 \end{enumerate}
\end{proof}
\begin{example}
We demonstrate the idea of the above proof in the following identity.
  \begin{align*}
    \sigma^B_{1526347}\cdot \sigma^B_{2314567}-\sigma^B_{2516347}\cdot\sigma^B_{1324567}+\sigma^B_{3516247}\cdot\sigma^B_{1234567}=0.
  \end{align*}
  Applying the determinantal formula, we see that
  \begin{align*}
    &\sigma^B_{1526347}=\det \left(
                          \begin{array}{cc}
                            H_3(X_2) & H_4(X_2) \\
                            H_1(X_4) & H_2(X_4) \\
                          \end{array}
                        \right), \qquad\qquad\qquad\sigma^B_{2314567}=\det \left(
                          \begin{array}{cc}
                            H_1(X_2) & H_2(X_1) \\
                            H_0(X_2) & H_1(X_1) \\
                          \end{array}
                        \right),\\
   & \sigma^B_{2516347}=\det\left(
                       \begin{array}{ccc}
                         H_1(X_1) & H_4(X_1) & H_5(X_1) \\
                         H_0(X_2) & H_3(X_2) & H_4(X_2) \\
                         H_{-2}(X_4) & H_1(X_4) & H_2(X_4) \\
                       \end{array}
                     \right), \quad \sigma^B_{1324567}=\det \left(
                          \begin{array}{cc}
                            H_1(X_2) & H_2(X_1) \\
                            H_{-1}(X_2) & H_0(X_1) \\
                          \end{array}
                        \right),\\
   &\sigma^B_{3516247}=\det\left(
                       \begin{array}{ccc}
                         H_2(X_1) & H_4(X_1) & H_5(X_1) \\
                         H_1(X_2) & H_3(X_2) & H_4(X_2) \\
                         H_{-1}(X_4) & H_1(X_4) & H_2(X_4) \\
                       \end{array}
                     \right), \quad \sigma^B_{1234567}=\det \left(
                          \begin{array}{cc}
                            H_0(X_2) & H_1(X_1) \\
                            H_{-1}(X_2) & H_0(X_1) \\
                          \end{array}
                        \right).
  \end{align*}
  We write
  \begin{align*}
      &\sigma^B_{1526347}=\det \left(
                          \begin{array}{cc}
                            H_3(X_2) & H_4(X_2) \\
                            H_1(X_4) & H_2(X_4) \\
                          \end{array}
                        \right)=\det \left(
                       \begin{array}{ccc}
                         1 & H_4(X_1) & H_5(X_1) \\
                         0 & H_3(X_2) & H_4(X_2) \\
                         0 & H_1(X_4) & H_2(X_4) \\
                       \end{array}
                     \right).
  \end{align*}
Notice that the last two columns of these $3\times 3$ matrix are the same, so it suffices to prove that
\begin{fontsize}{9pt}{12pt}
\begin{align*}
  &1\times \det\left(
                          \begin{array}{cc}
                            H_1(X_2) & H_2(X_1) \\
                            H_0(X_2) & H_1(X_1) \\
                          \end{array}
                        \right)- H_1(X_1) \det \left(
                          \begin{array}{cc}
                            H_1(X_2) & H_2(X_1) \\
                            H_{-1}(X_2) & H_0(X_1) \\
                          \end{array}
                        \right)+ H_2(X_1)\det \left(
                          \begin{array}{cc}
                            H_0(X_2) & H_1(X_1) \\
                            H_{-1}(X_2) & H_0(X_1) \\
                          \end{array}
                        \right)=0,\\
  & -H_0(X_2)\det \left(
                          \begin{array}{cc}
                            H_1(X_2) & H_2(X_1) \\
                            H_{-1}(X_2) & H_0(X_1) \\
                          \end{array}
                        \right)+ H_1(X_2)\det \left(
                          \begin{array}{cc}
                            H_0(X_2) & H_1(X_1) \\
                            H_{-1}(X_2) & H_0(X_1) \\
                          \end{array}
                        \right)=0.
\end{align*}
\end{fontsize}
These follow from the Laplace expansion of the following identities respectively.
\begin{align*}
 & \det \left(
         \begin{array}{ccc}
           H_2(X_1) & H_1(X_2) & H_2(X_1) \\
           H_1(X_1)& H_0(X_2) & H_1(X_1) \\
           1=H_0(X_1) & H_{-1}(X_2) & H_0(X_1) \\
         \end{array}
       \right)=0,\\
&\det \left(
         \begin{array}{ccc}
           H_1(X_2) & H_1(X_2) & H_2(X_1) \\
           H_0(X_2)& H_0(X_2) & H_1(X_1) \\
          0=H_{-1}(X_2)& H_{-1}(X_2) & H_0(X_1) \\
         \end{array}
       \right)=0.
\end{align*}
\end{example}

It remains in this section to deduce the identity~\eqref{e:keyIdentity} also in the partial flag variety setting. We use the following result due to Dale Peterson.


\begin{prop}[Proposition 11.1 in \cite{Rie03}]\label{prop:extension}
Let $w\in W$ and let $\sigma_w^B$ be the corresponding quantum Schubert class regarded as a function on the Peterson variety $\mathcal Y_{B_-}$ for the complete flag variety. Let $\widetilde\sigma_w^B$ be the rational function on the closure $\mathcal Y=\overline{\mathcal Y}_{B_-}$ that agrees with $\sigma_w^B$ on $\mathcal Y_B$. If $w\in W^P$, then $\widetilde{\sigma}_w$ restricts to a regular function on $\overline{\mathcal Y}_{P}\subset \mathcal Y$, and this restriction represents the quantum Schubert class $\sigma_w^P\in QH^*(G^\vee/P^\vee)$ associated to $w$.
\end{prop}

This proposition implies that any identity in quantum Schubert calculus for the complete flag variety $G^\vee/B^\vee=\mathbb{F}\ell_n$ involving only Schubert classes of the form $\sigma^B_w$ for $w\in W^P$ and without quantum parameters, holds also in $QH^*(G^\vee/P^\vee)$ with $\sigma_w^B$  replaced by $\sigma_w^P$. As a consequence we have the following corollary; namely we obtain  Theorem~\ref{keyIdentity}  in the case $n_j+n_{j+1}\leq n$.

\begin{cor}\label{c:prop-specialcase}
Let $n-n_{j+1}<i<n-n_j$ for some $1\leq j\leq r-1$ and $d:=i-(n-n_{j+1})$. Let $\Xi$ and $w_J$ be as defined in Definition~\ref{d:wJ}. Assume that $n_j+n_{j+1}\le n$. Set $\sigma_{w_J}:=0$ if $w_J$ is not defined.
Then the following identity holds in $QH^*(X)$,
\begin{align*}
  \sum\limits_{J\in \Xi}(-1)^{|J|} \sigma_{w_J}\sigma_{[1,n_j+d]\backslash J}=0.
\end{align*}
\end{cor}

\subsection{Proof of Lemma~\ref{uinthemiddle}}


\begin{lemma}\label{uinthemiddle-specialcase}  Suppose $n-n_{j+1}< i< n-n_j$ for some $1\leq j\leq r-1$. Assume additionally that $n_j+n_{j+1}\le n$. If $\b$ is a critical point of $\mathcal F_\mathbf q$, then
  \[
  u_{i,i+1}=-(G^{n_j}_1(b_-\dot w_0 B_-)+G^{n_{j+1}}_1(b_-\dot w_0 B_-)).
  \]
\end{lemma}

\begin{proof}
The statement is  a direct consequence of Corollary~\ref{c:prop-specialcase} combined with Lemma~\ref{equivalence}.
\end{proof}


\begin{defn}
We define a group involution on $G=GL_n(\mathbb C)$ using a combination of inverse, transpose and conjugation by $\dot w_0$,
\[g\mapsto \tau( g):=\dot w_0 (g\inv)^T\dot w_0^{-1}.\]
\end{defn}
Let $Q\supseteq B_-$ be the parabolic subgroup with $I^Q=n-I^P=\{n-n_r,\cdots,n-n_1\}$.
It is straightforward to check that our involution has the following properties.
\begin{enumerate}
\item  $\tau (P)=Q$ and $\tau( U_+)=U_+$.
\item $\tau(\dot w_P)=\dot w_Q$.
\item for $x\in U_+$ we have the relationship $\tau(x)_{i,i+1}=x_{n-i,n-i+1}$, for the entries just above the diagonal.
\end{enumerate}

\begin{lemma}\label{uinthemiddle-otherspecialcase}  Suppose $n-n_{j+1}< i< n-n_j$ for some $1\leq j\leq r-1$. Assume additionally that $n_j+n_{j+1}\ge n$. If $\b $ is a critical point of ${\mathcal F}_\mathbf q$, then
  \begin{equation}\label{e:otherspecialcase}
  u_{i,i+1}=-(G^{n_j}_1(b_-\dot w_0 B_-)+G^{n_{j+1}}_1(b_-\dot w_0 B_-)).
  \end{equation}
\end{lemma}

\begin{proof}
Since $\b\in B_-\cap U_+\dot w_P^{-1}\dot w_0U_+ $, we have that $\tau(\b)\in B_-\cap U_+\dot w_Q^{-1}\dot w_0 U_+ $. We can now apply Lemma~\ref{uinthemiddle-specialcase} to $\tau (\b)$, where we must replace $P$ by $Q$. Namely for $\tau (\b)$, Lemma~\ref{uinthemiddle-specialcase} says that, if $n_{j}=n-(n-n_j)< n-i< n-(n-n_{j+1})=n_{j+1}$, and $(n-n_j)+(n-n_{j+1})\le n$ (which is equivalent to our assumptions on $i$), then
\[
  \tau (u)_{n-i,n-i+1}=-(G^{n-n_j}_1(\tau(b_-)\dot w_0 B_-)+G^{n-n_{j+1}}_1(\tau(b_-)\dot w_0 B_-)).
  \]
  Recall that
  \[G_1^m(gB_-):=\frac{\Delta_{[m+1,n]}^{\{m\}\cup [m+2,n]}(g)}{\Delta_{[m+1,n]}^{[m+1,n]}(g)}.
  \]
  Now we deduce that
  \[
  G^{n-m}_1(\tau (b_-)\dot w_0 B_-)=G^{m}_1(b_-\dot w_0 B_-),
  \]
  using Jacobi's theorem. Moreover by property (3) above, we have $ \tau(u)_{n-i,n-i+1}=u_{i,i+1}$. Therefore the identity \eqref{e:otherspecialcase} holds.
\end{proof}

\bigskip

\begin{proof}[Proof of  Lemma~\ref{uinthemiddle} and Theorem~\ref{keyIdentity}]  Lemma~\ref{uinthemiddle} follows from the combination of Lemmas \ref{uinthemiddle-specialcase} and \ref{uinthemiddle-otherspecialcase}. 
 We showed in Lemma~\ref{equivalence},  that Theorem~\ref{keyIdentity} is true if and only if Lemma~\ref{uinthemiddle} holds. Since Lemma~\ref{uinthemiddle} has now been proved, we are done.
\end{proof}

\section{Appendix}
In this Appendix we give a translation of the Pl\"ucker coordinate formula for the superpotential $\mathcal{F}_-$ using Young diagrams.

For $1\leq k<n$, we consider the set of partitions inside $k\times (n-k)$ rectangle,  $$\mathcal{P}_{k, n}:=\{(\lambda_1, \cdots, \lambda_k)\in \mathbb{Z}^k\mid n-k\geq \lambda_1\geq \lambda_2\geq\cdots\geq \lambda_k\geq 0\}.$$
There is a bijection $${[n]\choose k}\to \mathcal{P}_{k, n};\quad J=(j_1, \cdots, j_k)\mapsto \lambda(J)=(j_k-k, \cdots, j_2-2, j_1-1).$$
Geometrically, we consider the $k\times (n-k)$ rectangle of $k(n-k)$ unit boxes. A positive path of such rectangle is a path starting from the lower left hand corner and moving either upward or to the right along edges, towards the upper right hand corner.
In particular,  a  Pl\"ucker coordinate  $p_{j_1\cdots j_k}$ is naturally viewed as the positive path that moves upwards precisely at the $j_1, j_2, \cdots, j_k$-th steps.
Moreover, the boxes above the positive path $p_J$ form the partition $\lambda(J)$. We therefore use the following notation convention
  $$p_J=p_\lambda=p^{(k)}_{{\rm YD}(\lambda)},$$
 where the superscript $(k)$ is used to indicate that the Young diagram $\mbox{YD}(\lambda)$ of the  partition $\lambda$ is inside $k\times (n-k)$ rectangle.
In particular,
$$p_{[k]}=p_{(0,\cdots, 0)}=p^{(k)}_{\emptyset}.$$

\begin{example} The Young diagrams of the partitions $(4, 4, 4)$ and $(3, 2, 0)$ in $\mathcal{P}_{3,7}$ are given as follows.
  \ytableausetup{mathmode,  boxsize=0.7em}
 $$
\begin{array}{ccc}
 \ydiagram{4,4,4} &  &\ydiagram{3,2,0}\\
  (4,4,4)& & (3,2,0)
\end{array}
$$
The Pl\"ucker coordinate $p_{146}$ for $Gr(3, 7)$  corresponds to
the partition $(3, 2, 0)$.
\end{example}

 By $(m^l, 0^{k-l})$  we mean  the partition $(m, \cdots, m, 0, \cdots, 0)\in \mathcal{P}_{k, n}$ with $l$ copies of $m$.
The   Young diagram ${\rm YD}(m^l, 0^{k-l})$ is an $l\times m$ rectangle $\square_{l\times m}$, and  ${\rm YD}(1, 0^{k-1})=\square$.
We call  $(m^l, 0^{k-l})$ a \textit{maximal partition} in $\mathcal{P}_{k, n}$ if $l=k$ or $m=n-k$ holds.
\begin{defn}
  Let $\lambda\in \mathcal{P}_{k, n}$ and $\nu\in \mathcal{P}_{k-a, n-a}$. We define

  $$
      p^{(k)}_{\square_{k\times m}, {\rm YD}(\lambda)}:=\begin{cases}p^{(k)}_{ {\rm YD}(m^k+\lambda)},&\mbox{ if }m^k+\lambda \in \mathcal{P}_{k, n},\\0,&\mbox{ otherwise};
  \end{cases}$$

   $$
      p^{(k)}_{\square_{a\times (n-k)}, {\rm YD}(\nu)}:=\begin{cases}p^{(k)}_{ {\rm YD}((n-k)^a, \nu)},&\mbox{ if }((n-k)^a, \nu) \in \mathcal{P}_{k, n},\\0,&\mbox{ otherwise}.
  \end{cases}$$
\end{defn}

\begin{defn} Let $k<l<n$ and $1\leq m<l-k$. We define
\begin{align*}& L(p^{(k)}_{\square_{k\times m}}\cdot p^{(l)}_{\square_{(l-m)\times (n-l)}})
  :=\sum_{\mu\leq m^k}(-1)^{|\mu|+km}  p^{(k)}_{{\rm YD}(\mu)}\cdot p^{(l)}_{\square_{(l-m)\times (n-l)}, {\rm YD}((m^k/\mu)^c)},
 \end{align*}
where
  $(m^k/\mu)^c\in \mathcal{P}_{m, m+k}$  denotes  the conjugate of  $(m-\mu_k, \cdots, m-\mu_1)$.
We define
\begin{align*}& L(p^{(k)}_{\square_{k\times m},\scriptscriptstyle\scriptscriptstyle\square}\cdot p^{(l)}_{\square_{(l-m)\times (n-l)}})
  :=\sum_{\mu\leq \lambda^{'}; \mu_1\neq m}(-1)^{|\mu|+km}  p^{(k)}_{{\rm YD}(\mu)}\cdot p^{(l)}_{\square_{(l-m)\times (n-l)}, {\rm YD}((\lambda'/\mu)^c)},
 \end{align*}
where $\lambda'=(m+1, m^{k-1})$;  $(\lambda'/\mu)^c\in \mathcal{P}_{m, m+k+1}$ denotes   the conjugate of  { $(m-\mu_k, \cdots, m-\mu_2, 0)$ if $\mu_1=m+1$, or   of $(m-\mu_k, \cdots, m-\mu_1,1)$ if $\mu_1<m$.}
\end{defn}

\begin{thm} In terms of the Pl\"ucker coordinates indexed by Young diagrams, 
  \begin{align*}
     \mathcal{F}_-&=\sum_{i=1}^{n_1-1}{p^{(n_1)}_{\square_{i\times (n-n_1)}, \scriptscriptstyle\scriptscriptstyle\square} \over p^{(n_1)}_{\square_{i\times (n-n_1)}}}
       +\sum_{j=1}^{r-1}\sum_{m=1}^{n_{j+1}-n_j-1} {L(p^{(n_j)}_{\square_{n_j\times m},\scriptscriptstyle\scriptscriptstyle\square}\cdot p^{(n_{j+1})}_{\square_{(n_{j+1}-m)\times (n-n_{j+1})}}) \over L(p^{(n_j)}_{\square_{n_j\times m}}\cdot p^{(n_{j+1})}_{\square_{(n_{j+1}-m)\times (n-n_{j+1})}})} \\
       &+\sum_{i=1}^{n-n_r-1}{p^{(n_r)}_{\square_{n_r\times i}, \scriptscriptstyle\scriptscriptstyle\square} \over p^{(n_r)}_{\square_{n_r\times i}}}+\sum_{j=1}^r
        {p^{(n_j)}_{\square}\over p^{(n_j)}_{\emptyset}}+\sum_{j=1}^r
        q_{n_j}{p^{(n_j)}_{\square_{n_j\times (n-n_j)}\setminus q_{n_j}}\over p^{(n_j)}_{\square_{n_j\times (n-n_j)}}}
  \end{align*}
  where $\square_{n_j\times (n-n_j)}\setminus q_{n_j}$ denotes the Young diagram obtained by   removing $n_j-n_{j-1}$ boxes from the last column of $\square_{n_j\times (n-n_j)}$ and removing
    $n_{j+1}-n_j$ boxes from the last row, with the removal of the box at the bottom-right corner double counted.
\end{thm}

\begin{proof}
   It suffices to   discuss the $S^{(j)}_{i}$-terms in Theorem~\ref{fminus}. (Other terms therein are direct translations to Young diagrams.)

   For the denominator $L(p^{(n_j)}_{\square_{n_j\times m}}\cdot p^{(n_{j+1})}_{\square_{(n_{j+1}-m)\times (n-n_{j+1})}})$ as above, where $m=i-n_j$, we define a map $\alpha :\{ J| J\in \binom{[i]}{m}\}\rightarrow \{\mu|\mu\leq m^{n_j}\}$ as follows: it sends $J=\{a_1,...,a_m\}$ to the Young diagram $\alpha(J)$ with $a_1,...,a_m$ steps horizontal. It follows directly that $\alpha$ is a bijection. It remains to check the following facts:\\
    \begin{enumerate}
      \item $J\in \binom{[\min\{i,\hat{i}\}]}{m}$ if and only if the join $\square_{(n_{j+1}-m)\times (n-n_{j+1})}, {\rm YD}((m^k/\alpha(J))^c)$ is inside the $n_{j+1}\times (n-n_{j+1})$ rectangle.
   \item For $J\in \binom{[\min\{i,\hat{i}\}]}{m}$, we have
  $ p_{J\cup [\hat i+1,n]}= p^{(n_{j+1})}_{\square_{(n_{j+1}-m)\times (n-n_{j+1})}, {\rm YD}((m^k/\alpha(J))^c)}$ and  $p_{[i]\smallsetminus J}=p^{(n_j)}_{{\rm YD}(\alpha(J))}$. In particular for $J=[m]$, the corresponding product is   the leading term $p^{(n_j)}_{\square_{n_j\times m}}\cdot p^{(n_{j+1})}_{\square_{(n_{j+1}-m)\times (n-n_{j+1})}}$.
    \end{enumerate}   By definition,  $J\in \binom{[\min\{i,\hat{i}\}]}{m}$ if and only if the numbering of the first m vertical steps of the Young diagram $J\cup [\hat i+1,n]$ are $a_1,...,a_m$ and $J\cup [\hat i+1,n]$ is inside the $n_{j+1}\times (n-n_{j+1})$ rectangle. Notice that { ${\rm YD}((m^{n_j}/\alpha(J))^c)$ is the Young diagram $(a_m-m,...,a_1-1)$. Thus when the join $\square_{(n_{j+1}-m)\times (n-n_{j+1})}, {\rm YD}((m^{n_j}/\alpha(J))^c)$ is inside} the $n_{j+1}\times (n-n_{j+1})$ rectangle, the numbering of its first $m$ vertical steps are exactly $a_1,...,a_m$, and hence coincides with the Young diagram of $\alpha(J\cup [\hat i+1,n])$.  Therefore in this case,   the Pl\"ucker coordinates are also identified.

   The arguments for the  numerators  are similar. Let $\lambda^{'}=(m+1,m^{n_j-1})$. Here we define $\alpha^{'} :\{ J| J\in \binom{[i+1]\smallsetminus i}{m}\}\rightarrow \{\mu|\mu\leq \lambda^{'}, \mu_1\neq m \}$ as follows: $\alpha^{'}$ sends $\{a_1,...,a_m\}$ to the (unique) Young diagram $\alpha^{'}(J)$ inside $\lambda^{'}$ with $[i+1]\smallsetminus \{i,a_1,...,a_m\}$ steps vertical and $\mu_1\neq m$. Such  map is  a bijection. Again we can similarly check the following facts:
   \begin{enumerate}
       \item    $J\in \binom{[\min\{i+1,\hat{i}\}]\smallsetminus i\}}{m}$ if and only if the join $\square_{(n_{j+1}-m)\times (n-n_{j+1})}, {\rm YD}((\lambda^{'}/\alpha^{'}(J))^c)$ is inside the $n_{j+1}\times (n-n_{j+1})$ rectangle.\\
       \item For $J\in \binom{[\min\{i+1,\hat{i}\}]\smallsetminus i\}}{m}$, we have $p_{[i-1]\cup \{i+1\} \smallsetminus J}=p^{(n_j)}_{{\rm YD}(\alpha^{'}(J))}$ and\\
   $ p_{J\cup [\hat i+1,n]}=p^{(n_{j+1})}_{\square_{(n_{j+1}-m)\times (n-n_{j+1})}, {\rm YD}((\lambda^{'}/\alpha^{'}(J))^c)}$.
   \end{enumerate}

    When $\alpha^{'}(J)_{1}=m+1$. Let $J=\{a_1,...,a_m\}$. $(\lambda^{'}/\alpha^{'}(J))^c$ is a partition given by the conjugate of $(m-\mu_k,...,m-\mu_2,0)$, and we have the fact that $a_m\neq i+1$ and ${\rm YD}((\lambda^{'}/\alpha^{'}(J))^c)$ is the Young diagram $(a_m-m,...,a_1-1)$. Thus when the join $\square_{(n_{j+1}-m)\times (n-n_{j+1})}, {\rm YD}((\lambda^{'}/\alpha^{'}(J))^c)$ is inside the $n_{j+1}\times (n-n_{j+1})$ rectangle, the numbering of its first $m$ vertical steps are exactly $a_1,...,a_m$ and hence it coincides with the Young diagram of $J\cup [\hat i+1,n]$. Therefore the Pl\"ucker coordinates are also identified. The argument about other parts is similar.
\end{proof}

  \ytableausetup{mathmode,  boxsize=0.3em}
\begin{example}
   For $F\ell_{2, 4; 7}$, we have
   \begin{align*}
     \mathcal{F}_-&={p_{27}\over p_{17}}+{p_{24}p_{1567}-p_{14}p_{2567}+p_{12}p_{4567}\over p_{23}p_{1567}-p_{13}p_{2567}+p_{12}p_{3567}}+{p_{2346}\over p_{2345}}+{p_{3457}\over p_{3456}}
             +{p_{13}\over p_{12}}+{p_{1235}\over p_{1234}}+q_2{p_{46}\over p_{67}}+q_4{p_{1467}\over p_{4567}}\\
        &= {p^{(2)}_{\ydiagram{5,1}} \over p^{(2)}_{\ydiagram{5}}}  + { p^{(2)}_{\ydiagram{2,1}} p^{(4)}_{\ydiagram{3,3,3}}-p^{(2)}_{\ydiagram{2}} p^{(4)}_{\ydiagram{3,3,3,1}}+p^{(2)}_{\emptyset} p^{(4)}_{\ydiagram{3,3,3,3}} \over  p^{(2)}_{\ydiagram{1,1}} p^{(4)}_{\ydiagram{3,3,3}}-p^{(2)}_{\ydiagram{1}} p^{(4)}_{\ydiagram{3,3,3,1}}+p^{(2)}_{\emptyset} p^{(4)}_{\ydiagram{3,3,3,2}}}+ {p^{(4)}_{\ydiagram{2,1,1,1}} \over p^{(4)}_{\ydiagram{1,1,1,1}}}+ {p^{(4)}_{\ydiagram{3,2,2,2}} \over p^{(4)}_{\ydiagram{2,2,2,2}}}+ {p^{(2)}_{\ydiagram{1}} \over p^{(2)}_{\emptyset}}+{p^{(4)}_{\ydiagram{1}} \over p^{(4)}_{\emptyset}}
       +q_2 {p^{(2)}_{\ydiagram{4,3}} \over p^{(2)}_{\ydiagram{5,5}}}+q_4{p^{(4)}_{\ydiagram{3,3,2}} \over p^{(4)}_{\ydiagram{3,3,3,3}}}.
   \end{align*}
 \end{example}

\end{document}